\documentclass[reqno]{amsart}
\usepackage{amsmath,amssymb,amsthm}
\usepackage[normalem]{ulem} 
\usepackage{enumitem} 
\usepackage{graphicx}
\newtheorem{theorem}{Theorem}
\newtheorem{proposition}[theorem]{Proposition}
\newtheorem{lemma}[theorem]{Lemma}
\newtheorem{corollary}[theorem]{Corollary}
\newtheorem{conjecture}[theorem]{Conjecture}
\theoremstyle{remark}
\newtheorem{remark}[theorem]{Remark}
\theoremstyle{remark}
\newtheorem{definition}[theorem]{Definition}
\theoremstyle{remark}

\theoremstyle{remark}
\newtheorem{property}[theorem]{Property}

\newcommand{\graph}{\operatorname{graph}}
\newcommand{\norm}[1]{\left\|#1\right\|}

\newcommand{\RR}{\mathbb{R}}

\newcommand{\dist}{\operatorname{dist}}

\newcommand{\Gammabar}{\textrm{\sout{$\phantom{.}$}$\!\!\Gamma$}}
\newcommand{\diam}{\operatorname{diam}}
\newcommand{\WC}{\mathcal W}
\newcommand{\BC}{\mathcal B}
\newcommand{\WB}{(\mathcal W,\mathcal B)}
\newcommand{\itemizeEqnVSpacing}{\rule{0pt}{1pt}\vspace*{-12pt}}
\renewcommand{\dim}{\operatorname{dim}}

\newcommand{\bdry}{\operatorname{bdry}}
\newcommand{\dt}{(\delta,t)}

\begin{document}
\title[Estimates for a Variable Coefficient Wolff
Maximal Function]{$L^3$ Estimates for an Algebraic Variable
Coefficient Wolff Circular Maximal Function}
\author[J. Zahl]{Joshua Zahl}%
\address{Department of Mathematics, UCLA, 520 Portola Plaza Box 951555, Los Angeles CA 90095-1555, USA}
\email{jzahl@math.ucla.edu}

\date{\today}
\begin{abstract}%
In 1997, Thomas Wolff proved sharp $L^3$ bounds for his circular maximal function, and in 1999, Kolasa and Wolff proved certain non-sharp $L^p$ inequalities for a broader class of maximal functions arising from curves of the form $\{\Phi(x,\cdot)=r\}$, where $\Phi(x,y)$ satisfied Sogge's cinematic curvature condition. Under the additional hypothesis that $\Phi$ is algebraic, we obtain a sharp $L^3$ bound on the corresponding maximal function. Since the function $\Phi(x,y)=|x-y|$ is algebraic and satisfies the cinematic curvature condition, our result generalizes Wolff's $L^3$ bound. The algebraicity condition allows us to employ the techniques of vertical cell decompositions and random sampling, which have been extensively developed in the computational geometry literature.
\end{abstract}
\maketitle%

\section{Introduction}
\subsection{Background} Consider the Wolff circular maximal
function
\begin{equation}\label{CircularMaximalFunctionDefn}
M^\delta f(r) =
\sup_{x}\frac{1}{|C^\delta(x,r)|}\int_{C^{\delta}(x,r)}|f|,
\end{equation}
where $C^{\delta}(x,r)$ is the $\delta$--neighborhood of the
circle centered at $x$ of radius $r$. In \cite{Wolff1}, Wolff
proved that for each $\epsilon>0$ there exists a constant
$C_\epsilon$ such that

\begin{equation}\label{L3Bound}
\norm{M^\delta f}_{L^3([1/2,1])}\leq
C_\epsilon\delta^{-\epsilon}\norm{f}_{L^3(\RR^2)},
\end{equation}
which in particular implies that every BRK set (a planar set
containing a circle of each radius $r\in[1/2,1]$) must have
Hausdorff dimension 2. It is not possible to omit the
$\delta^{-\epsilon}$ factor since if \eqref{L3Bound} held with
this factor omitted, it would imply that every BRK set had
strictly positive Lebesgue measure, and this is known to be false.
Wolff's result built off of his earlier work\footnote{While
\cite{Wolff3} was published after \cite{Wolff1}, \cite{Wolff3} was
written first.} (jointly with Kolasa) in \cite{Wolff3}, where he
proved the bound
\begin{equation}\label{weakerLpBound}
\norm{M^\delta f}_q\leq
C_{p,q}\delta^{-\frac{1}{2}(\frac{3}{p}-1)}\norm{f}_p,\ \ \
p<\frac{8}{3},\ q\leq 2p^\prime.
\end{equation}
Equation \eqref{weakerLpBound} can almost be obtained by
interpolating \eqref{L3Bound} with the trivial bound
\begin{equation}
\norm{M^\delta f}_\infty\leq C\delta^{-1}\norm{f}_1,
\end{equation}
though in doing so we pick up an additional
$C_\epsilon\delta^{-\epsilon}$ factor.

However, this earlier Kolasa-Wolff result applied not only to
circles but to any family of curves satisfying Sogge's cinematic
curvature condition first introduced in \cite{Sogge}; let $U$ be a
neighborhood of $(a,b)\in\RR^2\times\RR^2$ and $\Phi\colon U\to
\RR$ with $\Phi$ smooth. Then the family of curves\footnote{Note
that we are reversing the role of $x$ and $y$ from the notation of
\cite{Wolff3}.} $\Gammabar(x,r) = \{y\colon \Phi(x,y)=r\}$ is said
to satisfy the \emph{cinematic curvature} condition provided
\begin{itemize}%
\item
\itemizeEqnVSpacing%
\begin{equation}\label{cinematicCurvatureGradientCondition}
\nabla_y\Phi(a,b)\neq 0.
\end{equation}
$\phantom{.}$
\item%
\itemizeEqnVSpacing%
\begin{equation}\label{cinematicCurvatureCondition}%
 \det\Big(\nabla_x\left[\begin{array}{c}
e\cdot \nabla_y\Phi(x,y)\\
e\cdot \nabla_y\big(\frac{e\cdot
\nabla_y\Phi(x,y)}{|\nabla_y\Phi(x,y)|}\big)
\end{array}\right]\Big|_{(x,y)=(a,b)}\Big)\neq 0,
\end{equation}
where $e$ is a unit vector orthogonal to $\nabla_y\Phi(a,b)$.
While there are two potential choices of vector $e$, the two
choices only differ by a sign, so the veracity of
\eqref{cinematicCurvatureCondition} is independent of the choice
made.
\end{itemize}
Informally, the second condition is a quantitative version of the
statement that two distinct curves cannot be tangent to second
order---it guarantees that if two curves $\Gammabar$ and $\tilde
\Gammabar$ intersect at a point $x$, then their normal vectors at
$x$ or their curvature at $x$ (or both) must differ by at least
the distance between $\Gammabar$ and $\tilde\Gammabar$ in some
suitable metric.

 Let $\Gammabar^\delta(x,r)$ be the
$\delta$--neighborhood of $\Gammabar$. Define
\begin{equation}\label{CurveMaximalFunctionDefn}
M_\Phi^\delta f(r) = \sup_{x\in
U_1}\frac{1}{|\Gammabar^\delta(x,r)|}\int_{\Gammabar^{\delta}(x,r)}|f|,
\end{equation}
where $U_1$ is a sufficiently small neighborhood of $a$. Then
Kolasa and Wolff proved that for any $f$ supported in a
sufficiently small neighborhood of $b,$
\begin{equation}\label{MweakerLpBound}
\norm{M_\Phi^\delta f}_{L^q([1/2,1])}\leq
C_{p,q}\delta^{-\frac{1}{2}(\frac{3}{p}-1)}\norm{f}_p,\ \ \
p<\frac{8}{3},\ q\leq 2p^\prime.
\end{equation}
\subsection{New Results}
\begin{theorem}\label{theoremOne}
Let $\Phi$ be an algebraic function satisfying the cinematic
curvature conditions \eqref{cinematicCurvatureGradientCondition}
and \eqref{cinematicCurvatureCondition} at $(a,b)$ and let $U_1$
be a sufficiently small neighborhood of $a$. Then for all $f$
supported in a sufficiently small neighborhood of $b$ and for all
$\epsilon>0$, there exist a constant $C_\epsilon$ depending only
on $\epsilon$ and $\Phi$ such that for all $\delta>0$,
\begin{equation}\label{ML3Bound}
\norm{M_\Phi^\delta f}_{L^3([1/2,1])}\leq
C_\epsilon\delta^{-\epsilon}\norm{f}_{L^3(\RR^2)}.
\end{equation}
\end{theorem}
\begin{remark}
See Appendix \ref{realAlgGeoAppendix} for the definition of an
algebraic function and related concepts.
\end{remark}
\begin{remark} Theorem \ref{theoremOne} generalizes
\eqref{L3Bound}. Indeed, $\Phi(x,y)=|x-y|$ is clearly algebraic,
and by the rotational, translational, and scale invariance of
$\Phi$, in order to verify the cinematic curvature condition it
suffices to verify the condition at the point $a=(0,0),\ b=(1,0)$.
Then $e=(0,1)$ and the determinant in
\eqref{cinematicCurvatureCondition} is 1. Furthermore, if
\begin{equation}
\Phi(x,y)=|x-y|+P(x,y)
\end{equation}
for $P$ a smooth algebraic function with $\norm{P}_{C^3}$
sufficiently small, then $\Phi$ satisfies
\eqref{cinematicCurvatureCondition} uniformly in the choice of
$a,b\in[0,1]^2.$ Thus we obtain \eqref{ML3Bound} for any family of
smooth algebraically perturbed circles, provided the perturbation
is not too large.
\end{remark}
We shall prove Theorem \ref{theoremOne} by modifying Schlag's
arguments in \cite{Schlag}. These arguments rely on a key
incidence lemma for circles, which is proved by Wolff in
\cite{Wolff4}. This incidence lemma employs various bounds on the
behavior of circle intersections, which do not obviously hold for
the more general class of curves we are considering. Luckily, most
of the analogous statements were proved by Kolasa and Wolff in
\cite{Wolff3}, so Theorem \ref{theoremOne} can largely be obtained
by patching together previously known results.

The constraint that $\Phi$ be algebraic is quite restrictive and
is likely not optimal (indeed it is reasonable to conjecture that
it is completely unnecessary). However, this constraint allows us
to use a ``semi-cylindrical algebraic decomposition'' argument
from real algebraic geometry. We shall discuss in Section
\ref{RiemannianMetricIntersectionSection} some conjectures about
how the algebraic requirements can be weakened.
\subsection{Proof Sketch}
Through standard reductions, it suffices to prove a discretized
version of a bound on the adjoint of the maximal operator
$M_\Phi^\delta$. Roughly speaking, if we have a collection of
``tubes'' $\{\Gamma^\delta\}$ corresponding to curves with
$\delta$--separated radii (see \eqref{gammaDefn} below for the
definition of $\Gamma$), we need to control the area of the region
where many of these tubes overlap. This is Lemma
\ref{quantitativeMaximalFunctionBound} below.

In \cite{Schlag}, Schlag showed that
\eqref{ML3Bound} holds for families of
curves satisfying two conditions. The first is a bound
(\eqref{schlagThmControlOfIntersectionSize} below) on
$|\Gamma^\delta\cap\tilde\Gamma^\delta|$ (where here $|\cdot|$
denotes Lebesgue measure) provided we have control over how close
$\Gamma$ and $\tilde\Gamma$ are to each other in a suitable
parameter space and how close the two curves are to being tangent.

The second requirement, which is made precise in
\eqref{schlagThmNumberIncidencesControl} below, controls the
number of almost-tangencies that can occur between the elements of
$\WC$ and $\BC$ if $\WB$ is a $t$--bipartite pair. Informally, two
collections of curves $\WC$ and $\BC$ are called a $t$--bipartite
pair if every two curves in $\WC$ (resp $\BC$) are close in an
appropriate parameter space while those in $\WC$ are far from
those in $\BC$ (there are some additional technical requirements
that we shall gloss over here. The full details can be found in
Definition \ref{defnOfAtBipartitePair}). The requirement is a
quantitative analog of the incidence geometry result that $N$
circles in $\RR^2$ can have at most $C_\epsilon N^{3/2+\epsilon}$
tangencies between pairs of circles. The incidence geometry result
was proved in \cite{Clarkson}, and in \cite{Wolff4}, Wolff
obtained the quantitative analog that was then used in Schlag's
argument.

The bulk of this paper will be devoted to showing that families of
curves arising from algebraic defining functions $\Phi$ satisfy
the second requirement, i.e. that
\eqref{schlagThmNumberIncidencesControl} is true. Once this has
been established, one can run Schlag's arguments virtually
verbatim to obtain Theorem \ref{theoremOne}.

\subsection{Thanks}
The author would like to thank Javier P\'erez for pointing out typos in an earlier version of this manuscript. The author was supported in part by the Department of Defense through the National Defense Science \& Engineering Graduate Fellowship (NDSEG) Program.

\section{Definitions and Initial
Reductions}\label{DefnAndInitialReductionSection} First, let us
assume $U=U_1\times U_2$ with $U_1,U_2$ sufficiently small disks
centered at $a$ and $b$ respectively (the requirement that $U_1$
and $U_2$ be disks will be relevant---we need $U_2$ to be a
semi-algebraic set). In particular, by selecting $U_1,U_2$
sufficiently small we can assume that the cinematic curvature
conditions hold for every point $(x,y)\in U_1\times U_2$ with
uniform bounds on $\nabla_y\Phi$ and with the determinant in
\eqref{cinematicCurvatureCondition} bounded uniformly away from 0.

Throughout this paper, $C,C^\prime$, etc.~will denote constants
that are allowed to vary from line to line. We will say $X\lesssim
Y$ or $X$ is $O(Y)$ if $X<CY$ and $X\sim Y$ if $X\lesssim Y$ and
$Y\lesssim X$.

Fix $0<\alpha<C^{-1}\diam(U_2).$ For $x\in U_1, r\in [1/2,1]$, we
define
\begin{equation}\label{gammaDefn}
\Gamma(x_0,r_0)=\{y\in B(b,\alpha)\colon\Phi(x_0,y)=r_0\}.
\end{equation}
We shall call these sets \emph{$\Phi$--circles}, and if $\Gamma$
is a $\Phi$--circle then $\Gamma^\delta$ will denote its
$\delta$--neighborhood. If $\Gamma,\tilde\Gamma,$ etc.~are
$\Phi$--circles, then unless otherwise noted, $x_0,r_0$ and
$\tilde x_0,\tilde r_0$ will refer to their respective centers and
radii. The $\Phi$--circles defined here are strict subsets of the
sets $\Gammabar$ defined in the introduction. However, if the
function $f$ is supported on a sufficiently small neighborhood of
$b$ then we can define a maximal function analogous to
\eqref{CurveMaximalFunctionDefn} with $\Gamma$ in place of
$\Gammabar$, and the two maximal functions will agree. Thus we
shall henceforth work with curves $\Gamma$ defined by
\eqref{gammaDefn}.

We shall restrict our attention to those $\Phi$--circles $\Gamma$
with $x_0\in U_1,\ r_0\in(1-\tau,1)$ for $\tau$ a sufficiently
small constant which depends only on $\Phi$. By standard
compactness arguments, we can recover $L^p([1/2,1])$ bounds on
$M_\Phi$ from those on the ``restricted'' version of $M_\Phi$ by
considering the supremum over a finite number of scaled versions
of the function.

Using standard reductions (see e.g.~\cite{Schlag}), in order to
prove Theorem \ref{theoremOne} it suffices to prove the following
estimate.
\begin{lemma}\label{quantitativeMaximalFunctionBound}
For $\eta>0$ and $\delta$ sufficiently small depending on $\eta$,
let $\mathcal A$ be a collection of $\Phi$--circles with
$\delta$--separated radii, with each radius lying in $(1-\tau,1)$.
Then there exists $\tilde{\mathcal A}\subset\mathcal A$ with
$\#\tilde{\mathcal A}\geq\frac{1}{C}\# \mathcal A$ such that for
all $\Gamma\in\tilde{\mathcal A}$ and $\delta<\lambda<1$,
\begin{equation}
\Big|B(b,C^{-1}\alpha)\cap\{y\in \Gamma^\delta\colon\
\sum_{\tilde\Gamma\in\mathcal A}\chi_{\tilde\Gamma^\delta}(y)
>\delta^{-\eta}\lambda^{-2}\}\Big|\leq \lambda|\Gamma^\delta|.
\end{equation}
\end{lemma}

In \cite{Schlag}, Schlag took Wolff's combinatorial incidence
result from \cite{Wolff4} and used it in conjunction with an
induction on scales argument to prove the analogue of Lemma
\ref{quantitativeMaximalFunctionBound} (in \cite{Schlag}, this is
Lemma 8). In order to state Schlag's theorem, we first need some
additional definitions.

\begin{definition}
For $X\subset B(b,\alpha)$, we define
\begin{equation}\label{defOfDeltaX}
\Delta_X(\Gamma,\tilde\Gamma)=\inf_{\substack{y\in \overline
X\colon \Phi(x_0,y)=r_0\\\tilde y\in \overline X\colon\Phi(\tilde
x_0,\tilde y)=\tilde r_0}} |y-\tilde
y|+\Big|\frac{\nabla_y\Phi(x_0,y)}{\norm{\nabla_y\Phi(x_0,y)}}-\frac{\nabla_y\Phi(\tilde
x_0,\tilde y)}{\norm{\nabla_y\Phi(\tilde x_0,\tilde y)}}\Big|.
\end{equation}
\end{definition}
Crucially,
\begin{equation*}
\Delta_{B(b,C^{-1}\alpha)}(\Gamma,\tilde\Gamma)\geq
\Delta_{B(b,\alpha)}(\Gamma,\tilde\Gamma),
\end{equation*}
but there exists a finite family of translates $\{t_i\}\subset
\RR^2$ (the cardinality of the family depends only on $C$) so that
\begin{equation}
\inf_i \Delta_{B(b+t_i,C^{-1}\alpha)}(\Gamma,\tilde\Gamma)\leq
\Delta_{B(b,\alpha)}(\Gamma,\tilde\Gamma).
\end{equation}
In the example $\Phi(x,y)=|x-y|$, $\Delta_X(\Gamma,\tilde\Gamma)$
describes how ``far'' (in $(x_0,r_0)$ parameter space) we would
need to move $\Gamma$ so that $\tilde\Gamma$ and the newly moved
curve $\Gamma^\prime$ are incident at some point in $X$. Indeed,
if $\Phi(x,y)=|x-y|$ and $X=\RR^2$ then
$\Delta_X(\Gamma,\tilde\Gamma)=\big||x_0-\tilde x_0| - |r_0-\tilde
r_0|\big|,$ provided $x_0,\tilde x_0\in U_1$ with $\diam(U_1)$
sufficiently small so that in particular, the only way circles can
be tangent is if they are internally tangent.

Let
\begin{equation}
d(\Gamma,\tilde\Gamma)=|x_0-\tilde x_0|+|r_0-\tilde r_0|.
\end{equation}
$d(\cdot,\cdot)$ is a metric on the space of curves. Throughout
our arguments, the particular choice of metric will not be
important since we will not care about multiplicative constants.

\begin{definition}\label{defnOfAtBipartitePair}
Let $\mathcal W,\mathcal B$ be collections of $\Phi$--circles. We
say that $\WB$ is a \emph{$t$--bipartite pair} if
\begin{align}
&|r_0-\tilde r_0|\geq\delta\ &\textrm{for all}\
\Gamma,\tilde\Gamma\in\WC\cup\BC,\\
&d(\Gamma,\tilde\Gamma)\in (t,2t)\ &\textrm{if}\ \Gamma\in \WC,\ \tilde\Gamma\in\BC,\\
&d(\Gamma,\tilde\Gamma)\in (0,t)\ &\textrm{if}\
\Gamma,\tilde\Gamma\in \WC\ \textrm{or}\ \Gamma,\tilde\Gamma\in
\BC.
\end{align}
\end{definition}
\begin{definition}\label{defnOfARectangle}
A \emph{$\dt$--rectangle} $R$ is the $\delta$--neighborhood
of an arc of length $\sqrt{\delta/t}$ of a $\Phi$--circle
$\Gamma$. We say that a $\Phi$--circle $\Gamma$ is \emph{incident}
to $R$ if $R$ is contained in the $C_1\delta$ neighborhood of
$\Gamma$. We say that $R$ is of type $(\gtrsim\mu,\gtrsim\nu)$
relative to a $t$--bipartite pair $\WB$ if $R$ is incident to at
least $C\mu$ curves in $\mathcal W$ and at least $C\nu$ curves in
$\mathcal B$ for some absolute constant $C$ to be specified later.
\end{definition}
We are now able to state Schlag's result.
\begin{proposition}[Schlag]\label{schlagsThm}
Let $\mathcal A$ be a family of $\Phi$--circles with
$\delta$--separated radii that satisfy the following requirements:
\begin{enumerate}[label=(\roman{*}), ref=(\roman{*})]
\item\label{schlagThmItemOne}\itemizeEqnVSpacing
\begin{equation}\label{schlagThmControlOfIntersectionSize}
|\Gamma^\delta\cap \tilde \Gamma^\delta\cap
B(b^\prime,C^{-1}\alpha)|\lesssim\frac{\delta^2}{(d(\Gamma,\tilde\Gamma)+\delta)^{1/2}(\Delta_{B(b,\alpha)}(\Gamma,\tilde\Gamma)+\delta)^{1/2}}
\end{equation}
for any $b^\prime$ in a sufficiently small neighborhood of $b$.
\item\label{schlagThmItemTwo} For any $t$--bipartite pair $\WB$,
with $t>C\delta$ for an appropriate choice of $C$;
$\WC,\BC\subset\mathcal A;\ \#\WC=m;\ \#\BC=n;$ and for any
$\epsilon>0$, the number of $(\gtrsim\mu,\gtrsim\nu)$
$(t,\delta)$--rectangles is at most
\begin{equation}\label{schlagThmNumberIncidencesControl}
C_\epsilon
(mn)^{\epsilon}\Big(\Big(\frac{mn}{\mu\nu}\Big)^{3/4}+\frac{m}{\mu}+\frac{n}{\nu}\Big).
\end{equation}
\end{enumerate}
Then Lemma \ref{quantitativeMaximalFunctionBound} holds for the
collection $\mathcal A$.
\end{proposition}
\begin{proof}
The proof of this theorem can be found in \cite{Schlag}, Section
4. However, we need the following minor modifications.
\begin{itemize}
\item Schlag actually requires the bound
\begin{equation}\label{schlagRequiredBound}
|\Gamma^\delta\cap \tilde
\Gamma^\delta|\lesssim\frac{\delta^2}{\big(d(\Gamma,\Gamma)+\delta\big)^{1/2}\big(\Delta_{B(b,\alpha)}(\Gamma,\Gamma)+\delta\big)^{1/2}\phantom{\Big|}}.
\end{equation}
in place of \eqref{schlagThmControlOfIntersectionSize}. However,
\eqref{schlagRequiredBound} can be obtained from
\eqref{schlagThmControlOfIntersectionSize} by summing over
finitely
many translates of the ball $B(b,C^{-1}\alpha)$.
\item Schlag stipulates that Requirement \ref{schlagThmItemTwo} in
the above theorem hold for all values of $t$ and $\delta$, not
merely those for which $t>C\delta$. However, there are at most
$\lesssim \delta^{-2}$ ($\delta,t$)--rectangles incident to $\WB$,
and if $t<C\delta$ we can use this fact in place the bound from
\eqref{schlagThmNumberIncidencesControl}.\qedhere
\end{itemize}
\end{proof}
The next sections shall be devoted to proving that any
$\delta$--separated family of $\Phi$--circles satisfy the two
requirements from Proposition \ref{schlagsThm}. Once this has been
established we will have proved Theorem \ref{theoremOne}. The
first requirement will not present much difficulty; indeed, it was
already proved by Kolasa and Wolff in \cite{Wolff3}, and it is
Property \ref{intersectionSizeBounds} in Section
\ref{sectionCinematicCurvImplications} below. Thus the bulk of our
efforts will be devoted to proving that the second requirement is
satisfied. This will appear as Lemma
\ref{biPartitePairRectControlLemma} in Section
\ref{sectionCountingIncidences}.
\section{Algebraic Considerations}

Let $\Gamma=\Gamma(x_0,r_0)$ be a $\Phi$--circle and $X\subset
B(b,\alpha)$ an open semi-algebraic set of dimension 2 (see
Appendix \ref{realAlgGeoAppendix} for the definition of the
dimension of a semi-algebraic set); in our discussion below we
will only consider balls. For $w = (w_1,w_2,w_3)\in\RR^3$, let
\begin{equation}
\begin{split}
 V_{\Gamma,X,w}=\{(x,r,y)\in U_1\times(1-\tau,1)\times X\colon \Phi(x_0,y)-r_0=w_1,\\
\Phi(x,y)-r=w_2, \nabla_y\Phi(x_0,y)\wedge\nabla_y\Phi(x,y)=w_3&\}\label{defnOfV},
\end{split}
\end{equation}
where
\begin{equation*}
(z^{(1)},z^{(2)})\wedge(\tilde z^{(1)},\tilde
z^{(2)})=z^{(1)}\tilde z^{(2)}-z^{(2)}\tilde z^{(1)}.
\end{equation*}
$V_{\Gamma,X,w}$ should be thought of as the space of pairs $(\tilde\Gamma,y)$ with $\tilde\Gamma$ a $\Phi$--circle tangent to $\Gamma$ at the point $y\in X$. Intuitively, we can think of $w_1,w_2,w_3$ as being 0. However,
setting $w_1,w_2,$ $w_3=0$ might cause $V_{\Gamma,X,w}$ to fail to
have the correct dimension. Thus we shall choose a very small
``generic'' choice of $w_1,w_2,w_3$ which fixes this problem. This
will be elaborated upon in Lemma \ref{coneIsSemiAlgebraic}.

Let
\begin{equation}
S_{\Gamma,X,w}=\big(\pi_{(x,r)}V_{\Gamma,X,w}\big)\cap \{(x,r)\colon |x-x_0|>C\delta\}
\end{equation}
for an appropriately chosen $C$, where $\pi_{(x,r)}\colon
(x,r,y)\mapsto (x,r)$ is the projection operator. $S_{\Gamma,X,w}$ should be thought of as the set of $\tilde\Gamma$ that are incident to $\Gamma$ at \emph{some} point $y\in X$. In the example
where $\Phi(x,y)=|x-y|$, $S_{\Gamma,X,0}$ is a section of the
right-angled ``light cone'' with vertex $(x_0,r_0)\in\RR^3,$ i.e.
\begin{equation*}
S_{\Gamma,X,0}\subset\{(x,r)\colon |x-x_0|=|r-r_0|\}.
\end{equation*}
\begin{lemma}\label{coneIsSemiAlgebraic}
For an appropriate choice of $0\leq w_1,w_2,w_3<C^{-1}\delta$,
$S_{\Gamma,X}$ is a semi-algebraic set of bounded complexity.
Furthermore, if $X=B(b,\alpha)$ then $S_{\Gamma,X}$ has
(semi-algebraic) dimension 2.
\end{lemma}
\begin{proof}
We shall first show that if $w_1,w_2,w_3$ are chosen appropriately
then $V_{\Gamma,X,w}$ is a semi-algebraic set of codimension 3. It
suffices to show that the the defining functions in
\eqref{defnOfV} are algebraic functions whose zero-sets intersect
transversely. $\Phi(x_0,y)-r_0$ and $\Phi(x,y)-r$ are immediately
seen to be smooth and algebraic since $\Phi$ is smooth and
algebraic. The components of $\nabla_y\Phi(x_0,y)$ and
$\nabla_y\Phi(x,y)$ are smooth and algebraic since the partial
derivatives of a smooth algebraic function are smooth and
algebraic, and thus $\nabla_y\Phi(x_0,y)\wedge\nabla_y\Phi(x,y)$
is smooth and algebraic. The complexity of these functions is
clearly independent of the choice of $\Gamma$. Finally, by Sard's
theorem we can find $0\leq w_1,w_2,w_3<C^{-1}\delta$ such that
$(w_1,w_2,w_3)$ is a regular value of the map
\begin{equation*}
(x,r,y)\mapsto\big(\Phi(x_0,y)-r_0,\ \Phi(x,y)-r,\
\nabla_y\Phi(x_0,y)\wedge\nabla_y\Phi(x,y)\big).
\end{equation*}
For such a choice of values of $w_1,w_2,w_3$ we have that
$S_{\Gamma,X,w}$ has geometric codimension 3, and thus
semi-algebraic codimension 3, as desired (see Appendix
\ref{realAlgGeoAppendix} for a review of the relevant real
algebraic geometry).

By the Tarski-Seidenberg theorem, $\pi_{(x,r)}V_{\Gamma,X,w}$ is
semi-algebraic of bounded complexity, and thus so is
$S_{\Gamma,X,w}$. At this point, the dimension of the components
of $S_{\Gamma,X,w}$ could be 0,1, or 2. However, we shall show in
Corollary \ref{coneHasDimTwo} below that if $X=B(b,\alpha)$, then
$S_{\Gamma,X,w}$ is a smooth manifold of dimension 2 or 3, and
thus the components of $S_{\Gamma,X}$ are in fact of
(semi-algebraic) dimension 2.
\end{proof}
\begin{remark}
It is somewhat curious to note that in our proof, we use algebraic
considerations to show $\dim(S_{\Gamma,X,w})\leq 2$ and
differential geometric considerations to show
$\dim(S_{\Gamma,X,w})\geq 2$, and thus conclude that
$\dim(S_{\Gamma,X,w})=2$.
\end{remark}
\begin{definition}
Abusing notation slightly, we shall suppress the dependence of
$S_{\Gamma,X,w}$ on $w$, and we shall define $S_{\Gamma,X}$ to be
$S_{\Gamma,X,w}$ for an appropriate choice of $w$, the existence
of which is guaranteed by Lemma \ref{coneIsSemiAlgebraic}. None of
our arguments below will depend on the specific choice of $w$, and
all of the constants in the estimates below will be independent of
the choice of $w$, provided $|w|<C^{-1}\delta$ for a sufficiently
large constant $C$.
\end{definition}
We have defined $S_{\Gamma,X}$ and $\Delta_{X}$ so that
\begin{equation}
S_{\Gamma,X,0}=\{\Gamma^\prime\colon\Delta_X(\Gamma,\Gamma^\prime)=0\},
\end{equation}
and thus since $0\leq w_1,w_2,w_3\leq C^{-1}\delta$,
\begin{align}
&S_{\Gamma,X} \in \{\Gamma^\prime
\colon\Delta_X(\Gamma,\Gamma^\prime)=0\}+B(0,C^{-1}\delta),\label{surfaceContainmentOne}\\
&\{\Gamma^\prime\colon\Delta_X(\Gamma,\Gamma^\prime)=0\}\in
S_{\Gamma,X}+B(0,C^{-1}\delta),\label{surfaceContainmentTwo}
\end{align}
where the $+$ symbol denotes the Minkowski sum. These inclusions
are the key facts linking the algebraic and geometric
properties of $\Phi$. Lemma \ref{coneIsSemiAlgebraic} allows us to
use the technique of semi-cylindrical algebraic decompositions
(aka vertical algebraic decompositions) to decompose $\RR^3$ into
a collection of ``cells'' adapted to a collection of surfaces
$\{S_{\Gamma,X}\}$. Informally, a cell is an open subset of
$\RR^3$ whose boundary consists of pieces of the surfaces from the
collection $\{S_{\Gamma,X}\}$ as well as additional surfaces that
are added to guarantee that the cells have certain favorable
properties. More precisely we have the following result.
\begin{lemma}\label{constDescrComplexCellDecomp}
Let $\mathcal D$ be a collection of $\Phi$--circles, $\#\mathcal
D=N$. Then there exists an algorithm for creating a vertical
decomposition of $U_1\times(1-\tau,1)$ (recall that $U_1$ and
$\tau$ were specified in Section
\ref{DefnAndInitialReductionSection} and depend only on $\Phi$)
into $\lesssim N^3\log N$ open (in $\RR^3)$ cells $\{\Omega_i\}$
such that $U_1\times(1-\tau,1)$ is the union of sets of the
following types:
\begin{itemize}
\item cells,%
\item the dividing surfaces
$\{S_{\Gamma,B(b,\alpha)}\colon\Gamma\in\mathcal D\},$%
\item vertical walls: 2--dimensional semi-algebraic sets whose
projections under the map $\pi_x\colon (x,r)\mapsto x$ are
1--dimensional semi-algebraic sets.
\end{itemize}

The cells in this decomposition have the property that
\begin{equation}
\Omega\cap S_{\Gamma, B(b,\alpha)}=\emptyset\ \textrm{for all
cells}\ \Omega\ \textrm{and all}\ \Gamma\in\mathcal D.
\end{equation}
Furthermore, for each cell $\Omega$ in the decomposition, there is
a bounded number (6 will suffice) of dividing surfaces such that
$\Omega$ is one of the cells arising from the decomposition
algorithm applied to this subcollection of surfaces (i.e.~the
existence of the other $N-6$ surfaces is irrelevant if all we care
about is the cell $\Omega$).
\end{lemma}
\begin{proof}
This statement follows from the techniques developed by Chazelle,
Edelsbrunner, Guibas, and Sharir in \cite{Chazelle}.
Unfortunately, while Theorem \ref{constDescrComplexCellDecomp} is
claimed in \cite{Chazelle} and follows (with some effort) from the
methods described in Chapter 8 of \cite{sharir}, we are unaware of a
complete and detailed proof of Theorem
\ref{constDescrComplexCellDecomp} in the literature. The author
intends to present such a proof in his forthcoming PhD thesis. In
the interests of keeping this paper self contained, we will give a
brief expository sketch of the vertical algebraic decomposition in
Appendix \ref{cellDecompositionSection}.
\end{proof}
\begin{lemma}\label{randomSamplingLemma}
Let $\BC$ be a collection of $\Phi$--circles, $\#\BC=n$. Randomly
select (see Remark \ref{randomSelectionRem})  a subset $\mathcal
D\subset\BC$ with $\#\mathcal D=N<C^{-1}n$, and let
$\{\Omega_i\}_1^{M},\ M\leq N^3\log N$ be the cells from Lemma
\ref{constDescrComplexCellDecomp}. Then with high probability (see
Remark \ref{highProbRem}) we have that for each $i$,
\begin{equation}\label{notTooManySurfacesCutACell}
\#\{\Gamma\in\mathcal\BC\colon
S_{\Gamma,B(b,\alpha)}\cap\Omega_i\neq\emptyset\}\lesssim\frac{N\log
n}{n}.
\end{equation}
\end{lemma}
\begin{remark}\label{randomSelectionRem}
To obtain our random selection we shall take a uniformly
distributed random sample with replacement from $\BC$. However,
our algorithm will only work if the elements of the sample are all
distinct. By requiring that $N\leq \frac{1}{C}n$ for $C$
sufficiently large, we can ensure that this will occur with high
probability, so this assumption will not cause difficulty.
\end{remark}
\begin{remark}\label{highProbRem}
By ``high probability'' we mean that for any probability $P<1$ we
can select a choice of constant $C$ in the quasi-inequality in
\eqref{notTooManySurfacesCutACell} so that the decomposition
satisfies \eqref{notTooManySurfacesCutACell} with probability at
least $P$. Later in the proof of Theorem \ref{theoremOne} we shall
need the above decomposition to satisfy additional properties
which also occur with high probability (relative to another set of
constants that we can weaken at will). We can ensure that all of
these properties are simultaneously satisfied by requiring that
each of the properties are separately satisfied with sufficiently
high probability and using the trivial union bound.
\end{remark}
\begin{proof}
Lemma \ref{randomSamplingLemma} follows from Lemma
\ref{constDescrComplexCellDecomp} by the technique of random
sampling (see e.g.~\cite{Clarkson}). Again, we shall briefly
review this technique in Appendix \ref{cellDecompositionSection}.
\end{proof}
Lemma \ref{constDescrComplexCellDecomp} (which is only used to
prove Lemma \ref{randomSamplingLemma}) is the only place where
Lemma \ref{coneIsSemiAlgebraic} is used, and it is thus the only
place where we use the requirement that $\Phi$ be algebraic. We
shall discuss in Section \ref{RiemannianMetricIntersectionSection}
some conjectures about how to obtain Lemma
\ref{constDescrComplexCellDecomp} through other (less algebraic)
means, though our best attempts in this direction have thus far
yielded only provisional results.

{\bf Added 2/14/2012}: In a recent paper, the author has obtained an analogue of Lemma 13 using the discrete polynomial ham sandwich theorem of Guth and Katz in place of Lemma \ref{constDescrComplexCellDecomp}. With this new technique, the requirement that $\Phi$ be algebraic is no longer necessary, i.e.~Theorem \ref{theoremOne} is established for all defining functions $\Phi$ satisfying the cinematic curvature condition. See \cite{Zahl} for further details.
\section{Cinematic Curvature and its Implications}
\label{sectionCinematicCurvImplications}%
Many of Wolff's
arguments from \cite{Wolff1} rely on the local differential
properties of families of circles. The relevant properties are
captured by the notion of cinematic curvature defined in the
introduction. In \cite{Wolff3}, Kolasa and Wolff establish several
key properties of families of curves with cinematic curvature
which we shall recall below.
\begin{property}[Straightening out]\label{straighteningOut} Let $x_0\in U_1$. Then we can find a
diffeomorphism $\psi_{x_0}\colon U_2^\prime\to U_2$ and a choice
of $r_0=r_0(x_0)$ such that
\begin{equation*}
\Phi(x_0,\psi_{x_0}(y))-r_0=y^{(2)}
\end{equation*}
where $U_2^\prime$ is an appropriately chosen domain (which may no
longer be a disk). Furthermore for fixed $y_0$,
\begin{equation}\label{continuityofDiffeo}
\psi_{x_0}(y_0)\ \textrm{and}\ r_0(x_0)\ \textrm{are continuous
functions of}\ x_0.
\end{equation}
This is discussed on page 126 of \cite{Wolff3}. To simplify notation,
we shall say that $\Phi$ has been \emph{straightened out} around
$x_0$ if we (temporarily) replace the function $\Phi(x_0,\cdot)$
with $\Phi(x_0,\phi_{x_0}(\cdot))-r_0(x_0)$, i.e.~in
``straightened out'' coordinates, $\Phi(x_0,y)=y^{(2)}$. Note that
if we straighten out around $x_0$ then in this new coordinate
system $\Phi$ might no longer be algebraic. This will not pose any
problems to our analysis below; we shall only be straightening out
to simplify the proofs of certain diffeomorphism-invariant
statements, and the statement can then be ``pulled back'' to the
original (semi-algebraic) $\Phi$. This process may change some of
the constants involved in the relevant statements. However
\eqref{continuityofDiffeo} will guarantee that the constants are
worsened by at most a bounded amount so we can safely ignore this
problem.
\end{property}
\begin{property}[Derivative bounds]\label{cinematicDerivative} If we straighten out $\Phi$ at
$x_0$ then for $y\in B(0,\alpha)$,
\begin{equation}\label{equivalentCinematicCurvatureConditionOne}
|\partial_{y^{(1)}}\Phi(x,y)|+|\partial^2_{y^{(1)}}\Phi(x,y)|\sim
|x-x_0|,
\end{equation}
\begin{equation}\label{equivalentCinematicCurvatureConditionTwo}
|\partial_{y^{(2)}}\Phi(x,\psi_{x_0,r_0}(y))|\sim 1,
\end{equation}
where $\partial_{y^{(1)}}$ denotes the partial derivative in the
$y^{(1)}$--direction, etc. The constants in the quasi-equalities
above are uniform in all variables. Indeed, since the cinematic
curvature condition is diffeomorphism invariant,
\eqref{equivalentCinematicCurvatureConditionOne} and
\eqref{equivalentCinematicCurvatureConditionTwo} are equivalent to
the cinematic curvature condition. This is addressed in Equation (21) of \cite{Wolff3} and the surrounding discussion.
\end{property}
\begin{property}[Unique point of parallel normals]\label{tangencyPoint}%
Let $\Gamma,\tilde\Gamma$ be $\Phi$--circles with
\begin{equation*}
\Delta_{B(b,C^{-1}\alpha)}(\Gamma,\tilde\Gamma)\leq
{C^\prime}^{-1}|x_0-\tilde x_0|
\end{equation*}
for a sufficiently large constant $C^\prime$. Then there is a
unique point
\begin{equation*}
\xi=\xi(x_0,r_0,\tilde x_0)\in\Gamma\cap B(0,\alpha)
\end{equation*}
such that
\begin{equation}\label{xiDefn}
\nabla_y\Phi(x_0,\xi)\wedge \nabla_y\Phi(\tilde x_0,\xi)=0.
\end{equation}
Furthermore,
\begin{equation}\label{xiComparableDelta}
|\Phi(\tilde x_0,\xi)-\tilde
r_0|\lesssim\Delta_{B(b,C^{-1}\alpha)}(\Gamma,\tilde\Gamma),
\end{equation}
and
\begin{equation}\label{intersectionContainedInBall}
\Gamma\cap\tilde\Gamma\cap B(b,C^{-2}\alpha)\subset B\Bigg(\xi,
C\Big(\frac{\Delta_{B(b,C^{-1}\alpha)}(\Gamma,\tilde\Gamma)}{|x_0-\tilde
x_0|}\Big)^{1/2}\Bigg).
\end{equation}
Equations \eqref{xiComparableDelta} and
\eqref{intersectionContainedInBall} are Equations (26) and (27) in
\cite{Wolff3}.
\end{property}
\begin{property}[Appolonius-type bounds]\label{appolonius} Let $t>C\delta$. Fix three $\Phi$--circles
$\Gamma_1,\Gamma_2,\Gamma_3$, let $B_0=B(b,C^{-2}\alpha)$, and let
\begin{equation}\label{YDefn}
\begin{split}
Y=\Big\{\Gamma\colon &\Delta_{B(b,C^{-1}\alpha)}(\Gamma,\Gamma_i)<C_1\delta,\ i=1,2,3;\\
&d(\Gamma\cap B_0,\Gamma_i\cap B_0)>t,\ i=1,2,3;\\
&\Gamma^\delta\cap\Gamma_i^\delta\cap B_0\neq\emptyset,\ i=1,2,3;\\
&\dist(\Gamma^{C_1\delta}\cap\Gamma_i^{C_1\delta}B_0, \Gamma^\delta\cap\Gamma_j^\delta\cap B_0)>C_3\sqrt{\delta/t},\ i\neq j\Big\}.
\end{split}
\end{equation}
Informally, $Y$ is the collection of curves that are almost
tangent to each of the curves $\Gamma_1,\Gamma_2,\Gamma_3,$ with
the additional requirement that the three regions of
almost-tangency not be too close to each other.

If we identify $\Phi$--circles $\Gamma$ with points
$(x_0,r_0)\in\RR^3$ then
\begin{equation}\label{boundOnYDiameter}
Y\ \textrm{is the union of two sets, each of diameter}\ \lesssim
t.
\end{equation}
This is is Lemma 3.1(ii) in \cite{Wolff3}.
\end{property}
\begin{property} For three fixed curves $\Gamma_1,\Gamma_2,\Gamma_3$, and a
given curve $\Gamma=\Gamma(x_0,r_0)$, we say that $\Phi$ is
\emph{$\Gamma$--adapted} if there exists points $a_1,a_2,a_3,$
with $a_j\in \Gamma_j$ such that
\begin{equation*}
|a_j-\xi_j(x_0)|\leq C^{-1}\sqrt{\delta/t},
\end{equation*}
and
\begin{align*}
\Phi(x,a_1)&=0,\\
\nabla_x\Phi(x,a_2)&=(e\cdot(a_2-a_1))\beta
\end{align*}
for all $x$, where $e$ is a unit tangent vector to $\Gamma_1$ at
$a_1$, $\beta$ is a vector independent of $y$ with $|\beta|\sim
1$, and
\begin{equation*}
\xi_i(x_0)=\xi(x_i,r_i,x_0).
\end{equation*}
\begin{remark}
Informally, the notion of a $\Gamma$--adapted defining function is
a way of getting around the problem that we are forced to work
with a defining function $\Phi$, but we are actually interested in
its level sets $\{\Phi(x,\cdot)=r\}$. Thus we are free (within
certain constraints to be dealt with below) to modify $\Phi$
provided that our new defining function has the same level sets as
the old one. Choosing a $\Gamma$--adapted defining function
(provided a suitable one exists) simplifies many of the
technicalities in our estimates.
\end{remark}
Lemma 3.6 in \cite{Wolff3} tells us that if $\Gamma\in Y$ then by
pre-composing $\Phi$ with suitable diffeomorphisms, a
$\Gamma$--adapted defining function $\Phi$ exists which satisfies
uniform derivative bounds, and this function $\Phi$ has the same
level sets as our original $\Phi$ (i.e.~it gives rise to the same
$\Phi$--circles), so the corresponding maximal functions are
identical (the adapted defining function may not be algebraic, but
this will not affect our analysis).

Now, if $\Phi$ is $\Gamma$--adapted, define
\begin{equation}\label{defnOfG}
T(x) = \left(\begin{array}{cc} %
\nabla_x\Phi(x,\xi_1(x))&-1\\
\nabla_x\Phi(x,\xi_2(x))&-1\\
\nabla_x\Phi(x,\xi_3(x))&-1
\end{array}\right).
\end{equation}
Informally, if we fix a choice of $\Gamma$ and select a defining
function adapted to $\Gamma$, then for $x$ in a neighborhood of
$x_0$, $T(x)$ describes how changing $x$ affects how close
$\Gamma(x,r_0)$ is to being tangent with each of
$\Gamma_1,\Gamma_2,\Gamma_3$.

Lemma 3.8 in \cite{Wolff3} tells us that when restricted to each
connected component of $Y$ (individually), $T$ is boundedly
conjugate to its linear part, i.e.~if $\Gamma$, and $\tilde
\Gamma$ lie in the same connected component of $Y$, then
\begin{equation}\label{KolasaLemma38}
T(x_0)T(\tilde x_0)^{-1}=I+E(\tilde x_0),
\end{equation}
where (say) $\norm{E(\tilde x_0)}<1/100$. Furthermore, for the
same choice of $\Gamma,\tilde\Gamma$,
\begin{equation}\label{xiCannotWander}
|\xi_1(\tilde x_0)-\xi_1(x_0)|\lesssim\sqrt{\delta/t}.
\end{equation}
Equation \eqref{xiCannotWander} is a consequence of Equation (45) in \cite{Wolff3} once we note that if $\tilde\Gamma\in Y$ is in the same connected component as $\Gamma\in Y$, then since $T$ is boundedly conjugate to its linear part, $|T(x_0)(\tilde x_0-x_0,\tilde r_0-r_0)|<C\delta.$
\end{property}
\begin{property}[Bounds on intersection area]\label{intersectionSizeBounds} Let $\Gamma,\tilde\Gamma$ be $\Phi$
circles. Then
\begin{equation}\label{areaControl}
|\Gamma^\delta\cap\tilde\Gamma^\delta\cap
B(b,C^{-2}\alpha)|\lesssim\frac{\delta^2}{\big(d(\Gamma,\tilde\Gamma)+\delta\big)^{1/2}\big(\Delta_{B(b,C^{-1}\alpha)}(\Gamma,\tilde\Gamma)+\delta\big)^{1/2}},
\end{equation}
\begin{equation}\label{diameterControl}
\textrm{diam}(\Gamma^\delta\cap\tilde\Gamma^\delta\cap
B(b,C^{-2}\alpha))\lesssim
\Big(\frac{\Delta_{B(b,C^{-1}\alpha)}(\Gamma,\tilde\Gamma)+\delta}{d(\Gamma,\tilde\Gamma)+\delta}\Big)^{1/2}.
\end{equation}
This is Lemma 3.1(i) in \cite{Wolff3}.
\end{property}
As noted above, when $\Phi(x,y)=|x-y|$, then
$S_{\Gamma,B(b,\alpha)}$ is a section of the right-angled
light-cone with focus at $(x_0,r_0)$. We shall establish several
lemmas that show that certain key properties of light cones are
preserved when we consider the set $S_{\Gamma,B(b,\alpha)}$ for
$\Phi$ a general defining function satisfying the requirements
from Theorem \ref{theoremOne}.
\begin{lemma}\label{lemmaDistCompDelta}
Let $\Gamma,\tilde\Gamma$ be $\Phi$--circles with
\begin{equation}\label{centersFarApart}
\Delta_{B(b,C^{-1}\alpha)}(\Gamma,\tilde\Gamma)<{C^\prime}^{-1}|x_0-\tilde
x_0|.
\end{equation}
Then there exists $\Gamma^\prime$ with $x_0^\prime=\tilde x_0,\
|r_0^\prime-\tilde r_0|\lesssim
\Delta_{B(b,C^{-1}\alpha)}(\Gamma,\tilde\Gamma)$ such that
\begin{equation}\label{surfaceVerticalDistance}
\Gamma^\prime\in S_{\Gamma,B(b,\alpha)}.
\end{equation}
Furthermore,
\begin{equation}\label{distanceControlsDelta}
\Delta_{B(b,\alpha)}(\Gamma,\tilde\Gamma)\lesssim
\dist(S_{\Gamma,B(b,\alpha)},\tilde\Gamma)\lesssim
\Delta_{B(b,C^{-1}\alpha)}(\Gamma,\tilde\Gamma).
\end{equation}
\end{lemma}
\begin{remark}
Note that we have to use different sets $X$ in the subscript of
$\Delta$ on the right and left sides of
\eqref{distanceControlsDelta}. In the case where $\Phi(x,y)=|x-y|$
(and thus we can define $\Phi$ over (say) a large dilate of the
unit circle),
\begin{equation*}
\Delta_{(B(0,100))}(\Gamma,\tilde\Gamma)=\big||x_0-\tilde
x_0|-|r_0-\tilde r_0|\big|,
\end{equation*}
provided $\Gamma,\tilde\Gamma$ lie in suitably restricted sets,
and if two circles are nearly incident, we can always change one
of them slightly so that they are exactly incident. In the more
general case we are considering, however, it may not always be
possible to make two almost-incident curves exactly incident by
changing one of them slightly; it is possible that when we try to
move one of the curves to make the two curves incident, the
``point of incidence'' occurs outside the domain of definition of
$\Phi$ (and thus there is no point of incidence). Thus, we need to
be more careful about how we define incidence and
almost-incidence. This consideration will occur frequently in the
lemmas below, and it will significantly lengthen our analysis.
\end{remark}
\begin{proof}
By \eqref{surfaceContainmentOne} and
\eqref{surfaceContainmentTwo}, in order to obtain
\eqref{distanceControlsDelta}, it suffices to establish the
estimate
\begin{equation}\label{distanceControlsDeltaModified}
\Delta_{B(b,\alpha)}(\Gamma,\tilde\Gamma)\lesssim
\dist(S_{\Gamma,B(b,\alpha),0},\tilde\Gamma)\lesssim
\Delta_{B(b,C^{-1}\alpha)}(\Gamma,\tilde\Gamma).
\end{equation}

First, note that $\Delta_{B(b,\alpha)}(\cdot,\cdot)$ is jointly
smooth in both variables with uniformly bounded derivatives. Since
$\Delta_{B(b,\alpha)}(\Gamma,\tilde\Gamma)=0$ for $\tilde\Gamma\in
S_{\Gamma,B(b,\alpha)}$, we immediately obtain the first
inequality in \eqref{distanceControlsDeltaModified}. The second
inequality in \eqref{distanceControlsDeltaModified} follows from
\eqref{surfaceVerticalDistance}, which we shall now prove.

Straighten out $\Phi$ around $\tilde x_0$. From Property
\ref{tangencyPoint} of $\Phi$, there exists $\xi\in
B(b,\alpha)\cap\Gamma$ such that
\begin{equation}\label{curvesHaveSameSlope}
\nabla_y\Phi(x_0,\xi)\wedge\nabla_y\Phi(\tilde x_0,\xi)=0,
\end{equation}
i.e.~(in straightened out coordinates)
\begin{equation*}
\frac{\nabla_y\Phi(x_0,\xi)}{|\nabla_y\Phi(x_0,\xi)|}=(\pm 1,0),
\end{equation*}
and
\begin{equation*}
|\Phi(\tilde x,\xi)-\tilde r_0|\lesssim
\Delta_{B(b,C^{-1}\alpha)}(\Gamma,\tilde\Gamma),
\end{equation*}
where here and below the implicit constants are uniform in the
choice of $\Gamma,\tilde\Gamma$ provided \eqref{centersFarApart}
is satisfied uniformly. Thus if we select $x_0^\prime=\tilde x_0,\
r_0^\prime= \tilde r_0 + \Phi(\tilde x_0,\xi)$ then $\xi$ lies on
$\Gamma^\prime,$ which establishes
\eqref{surfaceVerticalDistance}.
\end{proof}
\begin{corollary}\label{coneHasDimTwo}
$S_{\Gamma,B(b,\alpha)}$ is a smooth manifold and
$\dim(S_{\Gamma,B(b,\alpha)})\geq 2$.
\end{corollary}
\begin{proof}
Let $(\tilde x_0, \tilde r_0)\in S_{\Gamma,B(b,\alpha)}$. Then for $C$ sufficiently large, $B(0,1/C)$ embeds into $S_{\Gamma,B(b,\alpha)}$ in a neighborhood of $(\tilde x_0,\tilde r_0)$ via the embedding $(x,r)\mapsto(x+\tilde x_0, r^\prime)$, where $r^\prime$ is as described in Lemma \ref{lemmaDistCompDelta}.
\end{proof}
\begin{corollary}\label{CylinderCorollary}
There exists $C_0$ such that for all $\Phi$--circles $\Gamma,$ all
$(x,r)\in S_{\Gamma,B(b,\alpha)}$, and all $t<C^{-1}|x-x_0|$,
\begin{equation}
\pi_x \big(S_{\Gamma,B(b,\alpha)}\cap\{(x^\prime,r^\prime)\colon
|x-x^\prime|<t, |r-r^\prime|<C_0t\}\big)=\{x^\prime\colon
|x-x^\prime|<t\},
\end{equation}
i.e.~the cylindrical section centered at $(x,r)\in
S_{\Gamma,B(b,\alpha)}$ of radius $t$ and height $Ct$ contains all
of (or possibly all of one of the sheets of)
$S_{\Gamma,B(b,\alpha)}$ confined to the corresponding truncated
cylinder.
\end{corollary}
\section{Counting incidences between bipartite pairs of curve families}
\label{sectionCountingIncidences} %
Recall the definition of a $t$--bipartite pair $\WB$, a
$\dt$--rectangle, and a rectangle of type
$(\gtrsim\mu,\gtrsim\nu)$ relative to $\WB$ (Definition
\ref{defnOfARectangle}).
\begin{definition}
We shall say that a $\dt$ rectangle $R$ is of type
$(\sim\mu,\sim\nu)$ if it is of type $(\gtrsim\mu,\gtrsim\nu)$,
but is neither of type $(\gtrsim C\mu,\gtrsim\nu)$ nor $(\gtrsim
\mu,\gtrsim C\nu)$ for some absolute constant $C$ which shall be
determined later.
\end{definition}
\begin{definition}
We say that two $\dt$--rectangles are \emph{close} if
there is a $(2\delta,t)$ rectangle containing both of them. We say that two $\dt$--rectangles are
\emph{comparable} if there is a $(C_0\delta,t)$--rectangle
containing both of them.
\end{definition}
For $\WB$ a $t$--bipartite pair with $t>C\delta$ and $X$ a set,
define
\begin{align*}
\mathcal{I}_{X}&=\{ (\Gamma,\tilde\Gamma)\in\WB\colon
\Delta_X(\Gamma,\tilde\Gamma)<\delta\},\\
\tilde{\mathcal{I}}_{X}&=\{ (\Gamma,\tilde\Gamma)\in\WB\colon
\Delta_X(\Gamma,\tilde\Gamma)<C\delta\},
\end{align*}
for some constant $C$ to be determined later, where we recall that
$\Delta_X$ is defined in \eqref{defOfDeltaX}.

We shall state and prove a series of lemmas that are analogous to
Lemmas 1.5--1.16 in \cite{Wolff4}. If the proof of a lemma is the
same as that of the corresponding lemma in \cite{Wolff4} we shall
omit it. Throughout the discussion below, $\WB$ is a
$t$--bipartite pair with $\#\WC=m,\ \#\BC=n$.
\begin{lemma}\label{lemma15}$\phantom{1}$
\begin{enumerate}[label=(\roman{*}), ref=(\roman{*})]
\item\label{lemma15Item1} If
$\Delta_{B(b,C^{-1}\alpha)}(\Gamma,\tilde\Gamma)<\delta$, then
there exists a $\dt$--rectangle $R\subset B(b,\alpha)$ such
that $\Gamma$ and $\tilde\Gamma$ are tangent to any
$\dt$--rectangle close to $R$.
\item\label{lemma15Item2} Conversely, if $\Gamma,\tilde\Gamma$ are
tangent to a common $\dt$--rectangle $R\in B(b,\alpha)$
then $\Delta_{B(b,\alpha)}(\Gamma,\tilde\Gamma)\leq C\delta$, and
if $\Gamma,\tilde\Gamma$ are tangent to comparable
$\dt$--rectangles $R,R^\prime\in B(b,\alpha)$ then
$\Delta_{B(b,\alpha)}(\Gamma,\tilde\Gamma)\lesssim\delta$.
\end{enumerate}
\end{lemma}
\begin{lemma}\label{lemma116}
Let $\Gamma\in\WC,\tilde\Gamma\in\BC$. Then there are at most
$O(1)$ incomparable $\dt$--rectangles $R\subset
B(b,\alpha)$ tangent to both $\Gamma$ and $\tilde\Gamma$.
\end{lemma}
\begin{proof}
Since $d(\Gamma,\tilde\Gamma)\sim t$, \eqref{areaControl} gives us
the bound
\begin{equation}\label{controlOfRectCurveIntersectionSize}
|B(b^\prime,C^{-1}\alpha)\cap
\Gamma\cap\tilde\Gamma|\lesssim\delta^{3/2}t^{-1/2}
\end{equation}
for all $b^\prime$ in a sufficiently small neighborhood of $b$.
Each $\dt$--rectangle has area $\sim \delta^{3/2}t^{-1/2}$
and incomparable $\dt$--rectangles are pairwise disjoint.
The lemma follows by applying
\eqref{controlOfRectCurveIntersectionSize} to $O(1)$ choices of
$b^\prime=b+t_i$.
\end{proof}
\begin{lemma}\label{lemma117}$\phantom{1}$
\begin{enumerate}[label=(\roman{*}), ref=(\roman{*})]
\item\label{lemma117Item1} Let $\mathcal R\subset B(b,\alpha)$ be
a collection of pairwise nonclose rectangles. Then
\begin{equation*}
\#\tilde{\mathcal I}_{B(b,\alpha)}\gtrsim
\#\{(R,\Gamma,\tilde\Gamma)\in\mathcal R\times\BC\times \WC\colon
\Gamma\ \textrm{and}\ \tilde\Gamma\ \textrm{are tangent to}\ R\}.
\end{equation*}
\item There exists a collection $\mathcal R$ of pairwise
incomparable $\dt$--rectangles $R\in B(b,\alpha)$ such that
\begin{equation*}
\# \mathcal I_{B(b,C^{-1}\alpha)}\lesssim
\#\{(R,\Gamma,\tilde\Gamma)\in\mathcal R\times\BC\times \WC\colon
\Gamma\ \textrm{and}\ \tilde\Gamma\ \textrm{are tangent to}\ R\}.
\end{equation*}
\end{enumerate}
\end{lemma}
\begin{proof}
The first statement is immediate. The second statement can be proved in the same way as Lemma 1.7 in \cite{Wolff4} with \eqref{xiDefn} and \eqref{xiComparableDelta} used in place of the analogous equations in \cite{Wolff4}.
\end{proof}
\begin{lemma}
Let $\Gamma_1,\Gamma_2,\Gamma_3$ be three $\Phi$--circles. Let
$\mathcal R$ be a collection of pairwise incomparable rectangles
$R\in B(b,\alpha)$ with the property that for each $R\in\mathcal
R$ there is a $\Phi$--circle $\Gamma$ such that:

\begin{itemize}
\item $d(\Gamma,\Gamma_i)\geq t,\ i=1,2,3.$ %
\item $\Gamma,\Gamma_1$ are tangent to $R$.%
\item There exist two $\dt$--rectangles $R_2,R_3\in
B(b,\alpha)$ such that $\Gamma$ and $\Gamma_i$ are tangent to
$R_i,\ i=2,3$ and such that $R_1,R_2,R_3$ are pairwise
incomparable.
\end{itemize}
Then $\#\mathcal R\lesssim1.$
\end{lemma}
\begin{proof}
We shall establish the proof with the additional restriction that
$R$ must lie in $B(b^\prime,C^{-2}\alpha)$ for $b^\prime$ in a
sufficiently small neighborhood of $b$. Once this has been
established, we can recover the full result by selecting $O(1)$
choices of $b^\prime$ such that
$B(b,\alpha)\subset\bigcup_{b^\prime}B(b^\prime,C^{-2}\alpha).$

Let $R\in\mathcal R$ and let $\Gamma$ be a $\Phi$--circle
satisfying the above conditions. Then we must have $\Gamma\in Y,$
where $Y$ is as defined in \eqref{YDefn}; indeed the above
requirements on $\Gamma$ are precisely those needed to ensure that
$\Gamma\in Y$. By \eqref{areaControl},
\begin{equation}\label{areaControlInLemma}
\Gamma\cap\Gamma_1\cap B(b^\prime,C^{-2}\alpha)\subset B(\xi(x_0,
r_0,x_1), C\delta^{1/2}t^{-1/2}).
\end{equation}
Now, let $\Gamma_0\in Y$ and let $\tilde\Phi$ be a
$\Gamma_0$--adapted defining function with the same level sets as
$\Phi$. Since $\tilde\Phi$ has the same level sets as $\Phi$ and
the gradient of $\tilde\Phi$ is comparable to that of $\Phi$, it
suffices to prove the lemma for $\tilde\Phi$. However, by
\eqref{xiCannotWander} we have that if $\Gamma$ is in the same
connected component of $Y$ as $\Gamma_0$ then
\begin{equation}\label{xiControlInLemma}
|\xi(x_1,r_1,x_0)-\xi(x_1,r_1,x)|\lesssim\sqrt{\delta/t}.
\end{equation}
Since $Y$ contains only two connected components,
\eqref{areaControlInLemma} and \eqref{xiControlInLemma} imply that
\begin{equation}\label{possibleIntersectingRegion}
\begin{split}
\bigcup_{(x_0,r_0)\in Y}\Gamma(x_0&,r_0)\cap\Gamma_1\cap
B(b^\prime,C^{-2}\alpha)\\
&\subset \Big(B(z_0, C\delta^{1/2}t^{-1/2})\cap \Gamma_1\Big)\cup
\Big(B(z_1, C\delta^{1/2}t^{-1/2})\cap \Gamma_1\Big),
\end{split}
\end{equation}
where $z_0,z_1$ are points in the two connected components of $Y$
respectively. In particular, the set on the right hand side of
\eqref{possibleIntersectingRegion} has measure $\lesssim
\delta^{3/2}t^{-1/2}$. Since every $R\in\mathcal R$ must lie in
this set, and pairwise incomparable rectangles must be disjoint,
we obtain $\#\mathcal R\lesssim 1$.
\end{proof}
\begin{lemma}\label{lemma19}
Let $\Gamma,\tilde\Gamma$ be $\Phi$--circles with
$d(\Gamma,\tilde\Gamma)=t>C\delta$ and $r_0\geq\tilde r_0$. Let
$R,\tilde R\in B(b,C^{-1}\alpha)$ be comparable
$\dt$--rectangles with $\Gamma,\tilde\Gamma$ tangent to
$R,\tilde R$ respectively. Then
\begin{enumerate}[label=(\roman{*}), ref=(\roman{*})]
\item\label{lemma19Item1} $\tilde\Gamma\cap B(b,C^{-1}\alpha)$ is
contained in the $C\delta$--neighborhood of
\begin{equation*} \{y\in
B(b,\alpha)\colon\Phi(x_0,y)\leq r_0\}.
\end{equation*}
\item\label{lemma19Item2} For any constant $A$ there is a constant
$C(A)$ such that the cardinality of any set of pairwise
incomparable $\dt$--rectangles $R\in B(b,C^{-1}\alpha)$
each of which is tangent to $\Gamma$ and intersects the
$A\delta$--neighborhood of
\begin{equation*}
\{y\in B(b,\alpha)\colon\Phi(\tilde x_0,y)\leq r_0\}
\end{equation*}
 does not exceed
$C(A)$.
\end{enumerate}
\end{lemma}
\begin{proof}
Straighten $\Phi$ around $x_0$. By Lemma
\ref{lemma15}.\ref{lemma15Item2}, with $\alpha$ replaced by
$C^{-1}\alpha$, we have
$\Delta_{B(b,C^{-1}\alpha)}(\Gamma,\tilde\Gamma)\leq
C^\prime\delta$. Thus if we choose the value of $C$ in the
statement of the lemma to be sufficiently large (depending on
$C^\prime$), then $|x_0-\tilde x_0|>
C^{\prime\prime}\Delta_{B(b,C^{-1}\alpha)}(\Gamma,\tilde\Gamma)$,
so by Property \ref{tangencyPoint} of cinematic curvature, there
exists a unique point $\xi(\tilde x_0,\tilde
r_0,x_0)\in\tilde\Gamma$ satisfying \eqref{xiDefn}, i.e.
\begin{equation*}
\nabla_y\Phi(\tilde x_0,\xi)=(0,\pm 1),
\end{equation*}
so $\xi^{(1)}$ is the point where the function
$y^{(1)}\mapsto\Phi(\tilde x_0, (y^{(1)},y^{(2)}))$ achieves its maximum in the domain
$(y^{(1)},y^{(2)})\in B(b,\alpha)$, where
$y^{(2)}=y^{(2)}(y^{(1)})$ is implicitly defined by
$(y^{(1)},y^{(2)}(y^{(1)}))\in\tilde\Gamma$ (we can verify without
difficulty that this is well-defined). By
\eqref{xiComparableDelta} (noting that in the straightened out
coordinate system, $\Gamma=\{y^{(2)}=0\}\cap U_2^\prime$),
\begin{align*}
\Phi(\tilde x_0, \xi)&\lesssim
\Delta_{B(b,C^{-1}\alpha)}(\Gamma,\tilde\Gamma)\\
&\lesssim\delta,
\end{align*}
and thus for an appropriate choice of $C$,
\begin{equation*}
\tilde \Gamma\cap U_2^\prime\subset \{y^{(2)}<C\delta\}.
\end{equation*}
Returning to our original coordinate system, this is Statement
\ref{lemma19Item1} of the lemma.

To obtain the second statement, note that by the same reasoning as
above,
\begin{equation}
\begin{split}
\Gamma^{C\delta}\cap \Big(\{y\in B(b,\alpha)\colon\Phi(\tilde x_0,y)\leq
\tilde r_0&\}+B(0,A\delta)\Big)\\
&\subset \Gamma^{C(A)\delta}\cap\tilde\Gamma^{C(A)\delta}\cap B(b,\alpha)
\end{split}
\end{equation}
for a suitable constant $C(A),$ where the $+$ in the above
equation denotes the Minkowski sum. The result then follows from
\eqref{areaControl} and the fact that incomparable rectangles are
disjoint.
\end{proof}
\begin{lemma}\label{lemma110}$\phantom{1}$
\begin{enumerate}[label=(\roman{*}), ref=(\roman{*})]%
\item \label{lemma10Item1} The cardinality of any set of
$(\sim\mu,\sim\nu)$ rectangles is
$\lesssim \frac{mn^{2/3}}{\mu\nu^{2/3}}$.%
\item \label{lemma10Item2} The cardinality of any set of
$(\gtrsim\mu,\gtrsim\nu)$ rectangles is
$\lesssim\frac{mn^{2/3}}{\mu\nu^{2/3}}+\frac{n}{\nu}\log\frac{m}{\mu}$.
\end{enumerate}
\end{lemma}
\begin{remark}
Recall that a rectangle of type $(\gtrsim \mu,\gtrsim \nu)$ is a
rectangle that is incident to at least $C\mu$ curves in $\WC$ and
at least $C\nu$ curves in $\BC$ for some absolute constant $C$ (a
rectangle of type $(\sim\mu,\sim\nu)$ is defined similarly), so
the statement of the lemma is well defined.
\end{remark}
\begin{proof}
Combined with the previous lemmas, Statement \ref{lemma10Item1} is just the graph-theo\-retic statement, due to K\H{o}vari, S\'os, and Turan in \cite{Turan}, that a $m\times n$ matrix with
entries 0 and 1 which has a forbidden $2\times 3$ submatrix of 1s
has $\lesssim mn^{2/3}$ 1s in total. Statement \ref{lemma10Item2} is obtained
from Statement \ref{lemma10Item1} by dyadic summation.
\end{proof}

The following lemma is the analogue of Lemma 1.11 in \cite{Wolff4}. The proof is identical.
\begin{lemma}\label{lemma111}
Let $\WB$ be a $t$--bipartite pair that has no
$(\gtrsim1, \gtrsim\nu_0)$ or $(\gtrsim\mu_0,\gtrsim1)$ rectangles
$R\in B(b,\alpha)$. Then
\begin{equation}
\#\mathcal I_{B(b,C^{-1}\alpha)}(\mathcal W,\mathcal
B)\lesssim\mu_0^{1/3}nm^{2/3}\log \nu_0+\nu_0m\log\mu_0.
\end{equation}
\end{lemma}
\begin{lemma}\label{cellDecompLemma3}
Let $\WB$ be a $t$--bipartite pair with $\#\BC=n$. Randomly
select a subset $\mathcal D\subset\BC$ with $\#\mathcal D
=N<\frac{1}{C}n$. (we shall call the elements of $\mathcal D$
\emph{dividing circles}), and let $\mathcal
S=\{S_{\Gamma,B(b,\alpha)}\colon\Gamma\in\mathcal D\}$. Then with
high probability (relative to our random selection of $\mathcal
D\subset\BC$), we can partition
\begin{equation}
\WC=\WC^*\sqcup\bigsqcup_1^M \WC_i
\end{equation}
so that the decomposition has the following properties.
\begin{enumerate}[label=(\roman{*}), ref=(\roman{*})]
\item \label{numberOfCells}$M\lesssim N^3\log N$.
\item \label{numberOfCellIncidentCircles}For each $i$,
\begin{equation*} \#\{\Gamma\in\BC\colon
\Delta_{B(b,C^{-1}\alpha)}(\Gamma,\tilde\Gamma)\leq C\delta\
\textrm{for some}\ \tilde\Gamma\in\WC_i\}\lesssim\frac{n\log
n}{N}.
\end{equation*}
\item \label{WStarProperties}For each $\Gamma\in\WC^*$ there
exists a dividing $\Phi$--circle $\tilde\Gamma$ such that
\begin{equation*}
\Delta_{B(b,\alpha)}(\Gamma,\tilde\Gamma)\lesssim\delta.
\end{equation*}
\end{enumerate}
\end{lemma}
\begin{remark}
The implicit constants appearing above depend only on $\Phi$ and
the probability that a randomly selected $\mathcal D\subset\BC$
has the desired properties. In particular, by worsening the
implicit constants we can make the probability arbitrarily close
to 1.
\end{remark}
\begin{proof}
Perform the cell decomposition of the arrangement $\mathcal D$, as
described in Lemma \ref{constDescrComplexCellDecomp}. Let
\begin{equation}
\mathcal W^*=\{\Gamma\in\WC\colon \dist(\Gamma,S_{\tilde\Gamma,
B(b,\alpha)})\leq C\delta\ \textrm{for some}\
\tilde\Gamma\in\mathcal D\},
\end{equation}
and for each $i=1,\ldots, M$, let
\begin{equation}
\mathcal W_i =\{\Gamma\in\WC\backslash\WC^*\colon \Gamma\in
\overline\Omega_i\}.
\end{equation}
If some $\Gamma$ is present in more than one $\WC_i$, remove it
from all but one of the $\WC_i$ (the choice is irrelevant). We
shall now verify that this decomposition satisfies the properties
claimed in the lemma. Property \ref{numberOfCells} is immediate from
Lemma \ref{constDescrComplexCellDecomp}, and Property
\ref{WStarProperties} follows from \eqref{distanceControlsDelta}. Thus it remains to verify Property \ref{numberOfCellIncidentCircles}. The idea is to show that if $\Gamma\in\BC$ satisfies $\Delta_{B(b,C^{-1}\alpha)}(\Gamma,\tilde\Gamma)\leq C\delta$ for some $\tilde\Gamma\in\WC_i$, then $\Gamma$ must lie in the corresponding cell $\Omega_i$ of the cell decomposition. Once this has been established we can use \eqref{notTooManySurfacesCutACell} to control the number of times this can occur.

Suppose $\Gamma\in\WC_i,\tilde\Gamma\in\BC$ with
$\Delta_{B(b,C^{-1}\alpha)}(\Gamma,\tilde\Gamma)\leq C\delta$. Then
by \eqref{distanceControlsDelta},
\begin{equation*}
\dist(\Gamma,S_{\tilde\Gamma,B(b,\alpha)})\leq C\delta,
\end{equation*}
and so we
can select $\Gamma^\prime\in S_{\tilde\Gamma,B(b,\alpha)}$ with
$d(\Gamma^\prime,\Gamma)\leq C\delta$ (for possibly a larger constant
$C$). Furthermore, since $\WB$ is a $t$--bipartite pair, we have
that $|x_0^\prime-\tilde x_0|\gtrsim t > C\delta$, and thus by Corollary
\ref{CylinderCorollary}, there exists $r_0^{\prime\prime}$ such that
\begin{equation}\label{defOfRelevantCylinder}
\Gamma(x_0^\prime,r_0^{\prime\prime})\in S_{\tilde\Gamma,
B(b,\alpha)}\cap\{(x^{\prime},r^\prime)\colon |x_0^\prime-\tilde
x_0|<C_1\delta, |r^{\prime\prime}-\tilde r_0|<C_2\delta\}.
\end{equation}
However, \eqref{defOfRelevantCylinder}
implies that $|r_0^{\prime\prime}-r_0|<C_2\delta$, and selecting constants
appropriately in the definition of $\WC^*$, this is less than
$\dist(\Gamma,S_{\Gamma^{\prime\prime\prime},B(b,\alpha)})$ for any
$\Gamma^{\prime\prime\prime}\in\mathcal D.$ Since the boundary of each cell
$\Omega$ consists only of dividing surfaces $S_{\Gamma^{\prime\prime\prime},
B(b,\alpha)}$ and vertical manifolds (2--dimensional surfaces that can be written as unions of vertical line segments), we conclude that
$(x_0^\prime,r_0^{\prime\prime})\in\Omega_i$, and thus
\begin{equation*}
S_{\tilde\Gamma,B(b,\alpha)}\cap \Omega_i\neq\emptyset.
\end{equation*}
Equation
\eqref{notTooManySurfacesCutACell} bounds the number of dividing
surfaces that can intersect each cell $\Omega_i$, and this in turn
gives us Property \ref{numberOfCellIncidentCircles}.
\end{proof}
\begin{lemma}\label{lemma113}
With high probability,
\begin{equation}
\#\mathcal W^*\lesssim\frac{n\# \tilde{\mathcal
I}_{B(b,\alpha)}(\mathcal W,\mathcal B)}{N}.
\end{equation}
\end{lemma}
\begin{proof}
This follows from Property \ref{WStarProperties} of Lemma
\ref{cellDecompLemma3}. Indeed, the probability of a given
$\Gamma\in\WC$ being in $\WC^*$ is bounded by
\begin{equation*}
\frac{n}{N}\{\tilde\Gamma\in\BC\colon\Delta_{B(b,\alpha)}(\Gamma,\tilde\Gamma)<C\delta\},
\end{equation*}
so the expected size of $\WC^*$ is $\frac{n\# \tilde{\mathcal
I}_{B(b,\alpha)}(\mathcal W,\mathcal B)}{N}$, from which the
result follows.
\end{proof}
\begin{definition}
We define a \emph{cluster} of $\Phi$--circles analogously to
Wolff's definition in \cite{Wolff4}: A cluster is a subset
$\mathcal C\subset\WC$ (or $\BC)$ with the property that there
exists a $\dt$--rectangle $R$ such that every
$\Gamma\in\mathcal C$ is tangent to a $\dt$--rectangle
comparable to $R$.
\end{definition}
\begin{lemma}\label{cardinalityOfClusterLemma}
Let $\mathcal C\subset\WC$ be a cluster and let $\Gamma\in\BC$.
Then then any set of pairwise incomparable
$\dt$--rectangles each of which is tangent to some circle
in $\mathcal C$ and to $\Gamma$ has cardinality $O(1)$.
\end{lemma}
\begin{remark}
Lemma \ref{lemma19} is used to prove this lemma. See Lemma 1.14 of \cite{Wolff4} for details.
\end{remark}
\begin{lemma}\label{lemma115}
Given a value of $\mu_0$, we can write
\begin{equation}
\WC=\WC_g\sqcup\WC_b,
\end{equation}
where
\begin{enumerate}[label=(\roman{*}), ref=(\roman{*})]
\item\label{lemma115Prop1} $\WC_g$ and $\BC$ have no
$\dt$--rectangles of type
$(\gtrsim\mu_0,\gtrsim1)$.%
\item\label{lemma115Prop2} $\WC_b$ is the union of
$\lesssim\frac{\#\WC}{\mu_0}(\log m)(\log n)$ clusters.
\end{enumerate}
\end{lemma}
\begin{lemma}\label{biPartitePairRectControlLemma}
Let $\WB$ be a $t$--bipartite pair with $m=|\mathcal W|,\
n=|\mathcal B|$. Let $\mathcal R$ be a set of pairwise
incomparable $(\geq\mu,\geq\nu)$ $\dt$--rectangles contained in $B(b,\alpha).$

For any $\epsilon>0$,
\begin{equation}\label{biPartitePairRectControl}
\#\mathcal R\lesssim_\epsilon
(mn)^{\epsilon}\Big(\Big(\frac{mn}{\mu\nu}\Big)^{3/4}+\frac{m}{\mu}+\frac{n}{\nu}\Big).
\end{equation}
\end{lemma}
In order to prove Lemma \ref{biPartitePairRectControlLemma}, it
suffices to consider the case where $\mu=\nu=1$ and establish the
bound
\begin{equation}\label{biPartitePairRectControlSpecialCase}
\#\mathcal R\lesssim_\epsilon
(mn)^{\epsilon}\Big((mn)^{3/4} + m\log n + n\log m\Big).
\end{equation}
To obtain \eqref{biPartitePairRectControl} from
\eqref{biPartitePairRectControlSpecialCase} we apply a random
sampling argument. The details of this random sampling argument
are on page 1253 of \cite{Wolff4}, so we shall not reproduce them here.
We shall call the $\Phi$--circles $\Gamma\in\mathcal W$ ``white''
$\Phi$--circles and those in $\mathcal B$ ``black''
$\Phi$--circles. By Lemma \ref{lemma116}, each pair
$(\Gamma,\tilde\Gamma)\in\mathcal\WB$ of white and black
$\Phi$--circles are jointly incident to at most $O(1)$
incomparable $\dt$--rectangles, so $\#\mathcal R\lesssim
mn$. Thus if $(mn)^{1/C}<\log(mn)$ then
\eqref{biPartitePairRectControlSpecialCase} holds immediately
(with an implicit constant depending on $C$). Thus we may assume
\begin{equation}\label{mnBig}
(mn)^{1/C}>\log(mn)
\end{equation}
for some fixed choice of $C$ which will be determined below.

We shall closely follow \cite{Wolff4} and substitute our lemmas
above for Wolff's analogous ones. Wolff's induction argument
allows him to control the number of incomparable
$\dt$--rectangles of type $(\gtrsim1,\gtrsim1)$ relative to
a collection $\WB$ over the region $B(b,\alpha)$ if he has similar
control over smaller collections $(\WC^\prime,\BC^\prime)$. Our
argument will allow us to control the number of incomparable
$(\gtrsim1,\gtrsim1)$ rectangles in a small region
$B(b,C^{-1}\alpha)$ if we have control over the number of
incomparable rectangles in a much larger region $B(b,\alpha),$ but
luckily we only require this control for smaller collections of
circles. Since the control is uniform in $b$, we can apply this
result to finitely many translates $\{b+t_i\}$ of $b$ to recover
the result over the larger region $B(b,\alpha)$, which allows us
to iterate the induction step. We shall focus on the key steps
where our arguments differ from Wolff's, and refer readers to
\cite{Wolff4} for the details of those arguments which are
identical.

To simplify our notation, we will employ the following definition:
\begin{definition}
For $\WB$ a $t$--bipartite pair and $X\subset\RR^2$, define $\mathcal R_X(\WC,\BC)$ to be the maximum possible cardinality of a set of pairwise incomparable rectangles of type $(\geq 1,\geq 1)$ that are contained in the set $X$.  
\end{definition}

Assume \eqref{biPartitePairRectControlSpecialCase} holds for all
pairs $(\WC^\prime,\BC^\prime)$ with
$(\#\WC^\prime)(\#\BC^\prime)<mn/2$. The base case of the
induction is taken care of by \eqref{mnBig}.

If $m\leq n^{\frac{1}{3}+\epsilon}$ or vice versa, then Lemma
\ref{biPartitePairRectControlLemma} follows from Lemma
\ref{lemma110}. Thus we may assume
\begin{equation}\label{mAndNComparable}
m^{1/3+\epsilon}<n<m.
\end{equation}
Let $\WC=\WC_g\cup\WC_b,\ \BC=\BC_g\cup\BC_b$ be the decomposition
from Lemma \ref{lemma115} with $\mu_0=\nu_0=(mn)^{1/4}.$ From
property \ref{lemma115Prop2} of the decomposition and Lemma \ref{cardinalityOfClusterLemma}, we have
\begin{align}
{\mathcal R}_{B(b,\alpha)}(\WC_b,\BC)<\log m\log
n(mn)^{3/4},\label{goodCircleControlOne}\\
{\mathcal R}_{B(b,\alpha)}(\WC,\BC_b)<\log m\log
n(mn)^{3/4}\label{goodCircleControlTwo}.
\end{align}
These quantities are $<\frac{1}{1000}(mn)^{\epsilon}(mn)^{3/4}$
provided that we choose the appropriate constant $C$ in
\eqref{mnBig}.

We shall now obtain the bound
\begin{equation}\label{controlOfBadInteractionsSmallCircle}%
{\mathcal R}_{B(b^\prime,C^{-1}\alpha)}(\WC_g,\BC_g)\leq
C_\epsilon (mn)^\epsilon (mn)^{3/4} C_0^{-1},
\end{equation}
where we can make $C_0$ arbitrarily large at the cost of
increasing $C_\epsilon$. Furthermore, this bound will be
independent of the choice of $b^\prime\in B(b,\alpha)$. Thus we
shall apply \eqref{controlOfBadInteractionsSmallCircle} with
$b^\prime = b+t_i$ for $\{t_i\}$ a finite family of translates
such that for every point $x\in B(b,\alpha),$ there exists an index $i$ such that $x$ is contained in $B(b+t_i,C^{-1}\alpha)$ and is distance at least $Ct$ from the boundary, and thus any $\dt$--rectangle contained in $B(b,\alpha)$ is contained in some $B(b+t_i,C^{-1}\alpha)$. We thus have
\begin{equation}
{\mathcal R}_{B(b,\alpha)}(\WC_g,\BC_g)\leq\sum
{\mathcal R}_{B(b+t_i,C^{-1}\alpha)}(\WC_g,\BC_g).
\end{equation}
Thus if we apply
\eqref{controlOfBadInteractionsSmallCircle} for each $t_i$ and
select $C_0$ sufficiently large we obtain
\begin{equation}\label{controlOfBadInteractionsAllCircle}%
{\mathcal R}_{B(b,\alpha)}(\WC_g,\BC_g)\leq \frac{1}{1000}C_\epsilon
(mn)^\epsilon (mn)^{3/4}.
\end{equation}
Combining \eqref{controlOfBadInteractionsAllCircle},
\eqref{goodCircleControlOne}, and \eqref{goodCircleControlTwo} and
using Lemma \ref{lemma117} we obtain
\eqref{biPartitePairRectControlSpecialCase}. It thus suffices to
prove \eqref{controlOfBadInteractionsSmallCircle}.

Write $\WC_g=\WC_g^*\sqcup\bigsqcup_1^M\WC_g^i$ as given by Lemma
\ref{cellDecompLemma3}, with $\alpha$ replaced by $C^{-1}\alpha$
and selecting a value of $N$ satisfying
\begin{equation}\label{controlOfN}
C\log(mn)^{1/\epsilon}< N <C^{-1}\min\Big(n^{3/4}m^{-1/4}\log(mn),
m^{1/4}n^{-1/12}\log(mn)\Big).
\end{equation}
Such a value of $N$ exists by assumption \eqref{mAndNComparable}
and by selecting a sufficiently large constant in \eqref{mnBig}.

We claim:
\begin{equation}\label{controlOfWgStar}
\#\WC_g^*\leq\frac{1}{1000C_0}\#\WC_b.
\end{equation}
Indeed, $(\mathcal \WC_g,\BC_g)$ contain no
$\dt$--rectangles of type $(\gtrsim \mu_0,\gtrsim 1)$ or
\mbox{$(\gtrsim 1,$} \mbox{$\gtrsim\nu_0)$} so by Lemma
\ref{lemma111} (with $\delta$ replaced by $C\delta$ for a suitable
constant $C$),
\begin{equation*}
\#\tilde{\mathcal I}_{B(b,C^{-1}\alpha)}(\WC_g,\BC_g)\lesssim
m^{5/4}n^{1/4}\log m + m^{3/4}n^{13/12}\log n,
\end{equation*}
and thus by Lemma \ref{lemma113} (recall that now $\alpha$ is
replaced by $C^{-1}\alpha$ and $C^{-1}\alpha$ is replaced by
$C^{-2}\alpha$) we can select our decomposition of $\WC_g$ so that
\begin{equation}
\#\WC_g^*\lesssim \frac{n}{N}\Big(m^{5/4}n^{1/4}\log m +
m^{3/4}n^{13/12}\log n\Big).
\end{equation}
Using \eqref{controlOfN} and selecting a sufficiently large
constant in \eqref{mnBig} (of course the choice of constant in
\eqref{mnBig} will depend on the desired constant $C_0$ in
\eqref{controlOfWgStar}) we obtain \eqref{controlOfWgStar}. Since
$(\#\WC_g^*)(\#\BC)<mn/2$ we can apply the induction hypothesis to
obtain
\begin{equation}\label{WStarAndB}
{\mathcal R}_{B(b,\alpha)}(\WC_g^*,\BC_g)\leq
\frac{1}{1000 C_0}C_\epsilon(mn)^{\epsilon}\Big((mn)^{3/4}+m\log
n+n\log m\Big).
\end{equation}

Now, for each $i$ let
\begin{equation}%
\BC_g^i=\{\Gamma\in\BC_g\colon
\Delta_{B(b,C^{-2}\alpha)}(\Gamma,\tilde\Gamma)<C\delta\
\textrm{for some}\ \tilde\Gamma\in\WC_g^i\}.
\end{equation}
Item \ref{numberOfCellIncidentCircles} in Lemma
\ref{cellDecompLemma3} implies
\begin{equation}
\#\BC_g^i\lesssim\frac{n\log n}{N}.
\end{equation}
Now, we can apply the induction hypothesis to the pair
$(\WC_g^i,\BC_g^i)$ to conclude
\begin{equation}\label{inductionBoundOneCellPreEqn}
\mathcal R_{B(b,\alpha)}(\WC_b^i,\BC_b^i)\leq C_\epsilon
[(\#\WC_b^i)(\#\BC_b^i)]^\epsilon[(\#\WC_b^i)(\#\BC_b^i)]^{3/4}
C_0^{-1}.
\end{equation}
However, $\BC_g^i$ was selected so that
\begin{equation*}
\mathcal R_{B(b,C^{-2}\alpha)}(\WC_g^i,\BC_g)\leq
\mathcal R_{B(b,\alpha)}(\WC_g^i,\BC_g^i),
\end{equation*}
and thus \eqref{inductionBoundOneCellPreEqn} implies
\begin{equation}\label{inductionBoundOneCell}
\mathcal R_{B(b,C^{-2}\alpha)}(\WC_g^i,\BC_g)\leq
C_\epsilon
[(\#\WC_g^i)(\#\BC_g^i)]^\epsilon[(\#\WC_g^i)(\#\BC_g^i)]^{3/4}
C_0^{-1}.
\end{equation}

Summing \eqref{inductionBoundOneCell} over the $M\lesssim N^3\log
N$ choices of $i$ and applying H\"older's inequality (see pages 1252--3
of \cite{Wolff4}, for the details), we obtain
\begin{equation}\label{controlOfAllCellInteractions}
\begin{split}
\sum_i\mathcal R_{ B(b,C^{-2}\alpha)}&(\WC_g^i,\BC_g^i)\\
& \leq
\frac{1}{1000}C_\epsilon(mn)^{\epsilon}\Big((mn)^{3/4}+m\log n+n\log m\Big).
\end{split}
\end{equation}
Combining \eqref{controlOfAllCellInteractions},
\eqref{controlOfN}, and \eqref{WStarAndB} we obtain
\eqref{controlOfBadInteractionsSmallCircle}.
%
\section{Riemannian metric circles and other generalizations}\label{RiemannianMetricIntersectionSection}%
It is reasonable to ask whether \eqref{ML3Bound} holds for
functions $\Phi$ which satisfy the cinematic curvature conditions
but are not algebraic. An examination of the arguments above
reveals that the only place where the algebraic properties of
$\Phi$ are used is in Lemma \ref{constDescrComplexCellDecomp},
where we make use of the fact that the level sets of
$\Phi(x,\cdot)$ (and of various functions obtained from $\Phi$)
are algebraic curves, and in particular, any two such curves
intersect $O(1)$ times.

One might hope that we could extend \eqref{ML3Bound} to analytic
$\Phi$ by approximating $\Phi$ by the first $\sim|\log\delta|$
terms of its Taylor expansion. Unfortunately, the bounds obtained
above are more than superexponential in the degree of $\Phi$, so
if we approximate $\Phi$ by a polynomial of degree
$\sim|\log\delta|$ then the above proof yields maximal function
bounds that are worse than the Kolasa-Wolff result
\eqref{MweakerLpBound}.

Working through the proof of Lemma
\ref{constDescrComplexCellDecomp}, we see that the proof requires
us to control the number of times certain pairs of curves can
intersect. For $x,\tilde x\in U_1,\omega\in\{\pm1\}$, let
\begin{equation}
\gamma_{x,\tilde x,\omega,r}=\{y\colon\Phi(x,y)+\omega\Phi(\tilde
x,y)=r\}.
\end{equation}
We shall call such curves $\Phi$--conics.
\begin{definition}
We say that $\Phi$ has the \emph{bounded conic intersection
property} if it satisfies the following requirements:
\begin{enumerate}[label=(\roman{*}), ref=(\roman{*})]
\item If $\{x,\tilde x\}\neq\{x^\prime,\tilde x^\prime\},$ then
\begin{equation}\label{boundedEllipseIntersection}
\#(\gamma_{x,\tilde x,\omega,r}\cap \gamma_{x^\prime,\tilde
x^\prime,\omega^\prime,r^\prime})\lesssim 1.
\end{equation}%
\item All $\Phi$--circles $\Gamma$ and $\Phi$--conics $\gamma$
have $O(1)$ $y^{(1)}$--extremal points (defined below).
\end{enumerate}
\end{definition}
\begin{definition}\label{extremalPtDefinition}
A $y^{(1)}$--extremal point of a curve $\zeta$ is a point
$y_0\in\zeta$ such that $\zeta\cap V$ is contained in one of the
closed half-spaces $\{y^{(1)}\geq y_0^{(1)}\}$ or $\{y^{(1)}\leq
y_0^{(1)}\}$ for $V$ a sufficiently small open neighborhood of
$y_0$.
\end{definition}

Requirement \eqref{boundedEllipseIntersection} is the most
difficult to satisfy, and it is the analogue of the Euclidean
statement that distinct irreducible conic sections intersect in at
most $O(1)$ places (actually 4).

If $\Phi$ satisfies the cinematic curvature hypotheses, it need
not have the bounded conic intersection property. Indeed, consider the example
\begin{equation}\label{ellipseIntersectionCounterExample}
\Phi(x,y)=y^{(2)}+x^{(1)}y^{(1)} + x^{(2)}(y^{(1)})^2 +p(x,y).
\end{equation}
If $p(x,y)=0$, the $\Phi$--conics
\begin{align*}
\gamma&=\{y\colon \Phi((1,0),y)+\Phi((-1,0),y)=r\},\\
\tilde\gamma&=\{y\colon \Phi((0,1),y)+\Phi((0,-1),y)=r\}
\end{align*}
are identical (both are simply the line $y^{(2)}=r$. Thus we can
select $p$ to be a highly oscillatory $C^\infty$ perturbation
which causes $\#(\gamma\cap\tilde\gamma)$ to be arbitrarily large,
independent of (say) the $C^3$--norm of $\Phi$ (we could choose
some other reasonable norm on $\Phi$ and construct similar
counter-examples). For example, we could choose
\begin{equation}
p(x,y)=C^{-1}\phi(x)\Big(y^{(2)}-\exp\left[-1/|(y^{(1)})^6|\right]\sin\left(\exp\left[1/|(y^{(1)})^2|\right]\right)\Big)
\end{equation}
for $\phi(x)$ a $C^\infty$ function supported in a small
neighborhood of $(1,0)$. This choice of $\Phi$ satisfies the
cinematic curvature hypothesis, since it satisfies
\eqref{equivalentCinematicCurvatureConditionOne} and
\eqref{equivalentCinematicCurvatureConditionTwo} (provided we
choose $C$ sufficiently large so the contributions from $p$ do
not affect the calculations), but it does not satisfy
\eqref{boundedEllipseIntersection}. Of course, the $\Phi$ given in
\eqref{ellipseIntersectionCounterExample} may still satisfy
\eqref{ML3Bound}, but a different proof would be needed. 

{\bf Added 2/14/2012}: Indeed, the new results from \cite{Zahl} show that the defining function $\Phi$ from \eqref{ellipseIntersectionCounterExample} satisfies the bound \eqref{ML3Bound}, though of course $\Phi$ from \eqref{ellipseIntersectionCounterExample} does not have the bounded conic intersection property.

While general $\Phi$ need not satisfy
\eqref{boundedEllipseIntersection}, we conjecture:
\begin{conjecture}\label{riemannianMetricsAreGood}
Let $\Phi(x,y)=\rho(x,y)$ for $\rho$ a Riemannian metric
sufficiently close to Euclidean. Then $\Phi$ satisfies the bounded
conic intersection property.
\end{conjecture}
This would imply
\begin{corollary}[conditional on conjecture
\ref{riemannianMetricsAreGood}]\label{riemannianMetricMaximalFn}
Let $\Phi(x,y)$ be as in Conjecture
\eqref{riemannianMetricsAreGood}. Then \eqref{ML3Bound} holds for
$M_\Phi$.
\end{corollary}
\begin{remark}
Actually, we can still obtain Corollary
\ref{riemannianMetricMaximalFn} if we weaken Conjecture
\ref{riemannianMetricsAreGood} to the following statement: If
$\Phi(x,y)=\rho(x,y)$ for $\rho$ a Riemannian metric, define a
$\delta$--generic $\Phi$--conic to be a curve $\gamma_{x,\tilde
x,\omega,r}$ which is not contained in the $\delta$--neighborhood
of any geodesic (this is a quantitative analogue of an (algebraic)
conic section being irreducible). Then $\gamma_{x,\tilde
x,\omega,r}$ admits a decomposition $\gamma_{x,\tilde
x,\omega,r}=\cup_i \gamma_{x,\tilde x,\omega,r}^i$ into
$\lesssim|\log\delta|^C$ connected components such that
\eqref{boundedEllipseIntersection} is satisfied for any two
components of any two $\Phi$--conics.
\end{remark}
\appendix
\section{The Cell Decomposition}\label{cellDecompositionSection}
We shall give a brief sketch of the techniques developed by
Chazelle et al.~in \cite{Chazelle} (see also \cite{Clarkson} and
\cite{sharir} for a rigorous exposition closer to the one sketched
here) on the method of vertical cell decompositions and random
sampling.

Let $\mathcal S=\{S_1,\ldots,S_N\}$ be a collection of
2--dimensional semi-algebraic sets in $\RR^3$ (for which we shall
use the coordinates $(x,r)\in\RR^2\times\RR$).

By subdividing each $S_i$ into a bounded number of pieces if
necessary, we may assume that each set $S_i$ may be written in one
of the following three forms:
\begin{itemize}
\item $S=\graph(f)$, for $f\colon V\to\RR$ a smooth algebraic
function and $V\subset\RR^2$ a (Euclidean) open semi-algebraic set. We shall
call these sets ``surface patches''. %
\item $S_i$ a semi-algebraic set with $\dim(S_i)=2$ but
$\dim(\pi_x(S_i))=1$. We shall call these sets ``vertical
manifolds.''%
\item $S_i$ a semi-algebraic set with $\dim(S_i)<2$.
\end{itemize}
To keep our exposition brief, we shall ignore the latter two types
of sets, since their presence is merely a technical annoyance that
does not contribute significantly to the analysis of the
decomposition. Thus we shall assume that the sets in $\mathcal S$
consist entirely of surface patches.
\begin{definition}
For $S$ a surface patch, we shall define $\bdry(S)=\overline
S\backslash S,$ where $\overline S$ denotes the closure of $S$ in
the Euclidean (rather than Zariski) topology. Note that
$\dim(\bdry(S))=1$.
\end{definition}

\begin{definition} A \emph{vertical line segment} $L\subset\RR^3$ is a
connected 1--dimensional semi-algebraic set with the property that
$\pi_x(L)$ is a point. If $(x_0,r_0)\in\RR^3$, we say that the
(connected) vertical line segment $L$ containing $(x_0,r_0)$ is
maximal with respect to $\mathcal S$ if $L$ meets no point of any
surface in $\mathcal S$ except possibly at $(x_0,r_0)$, but any
strictly larger line segment does.
\end{definition}

If $\gamma\subset\RR^3$ is a 1--dimensional semi-algebraic set
(i.e.~a union of segments of algebraic curves) which is not a
union of vertical lines and isolated points, then if we erect a maximal line segment
from every point of $\gamma$ we obtain a 2--dimensional
semi-algebraic set $V_\gamma$ with
$\pi_x(V_\gamma)=\pi_x(\gamma)$. We shall call this set the
``maximal vertical wall above $\gamma$'' (relative to $\mathcal
S$).

To construct the cell decomposition, erect a maximal vertical wall
above $S\cap \tilde S$ for every pair of distinct $S,\tilde
S\in\mathcal S$, and a maximal vertical wall above $\bdry(S)$ for
each $S\in\mathcal S$. If we consider $\RR^3$ with the surfaces
$S\in\mathcal S$ and the above maximal vertical walls removed,
then the remaining connected sets (which we shall call pre-cells)
each have a unique ``top'' and ``bottom'' bounding surface,
i.e.~for each pre-cell $\Omega$ there are unique $S,\tilde
S\in\mathcal S$ such that any maximal line containing
$(x,r)\in\Omega$ terminates at points in $S$ and $\tilde S$. Thus
at this point, each pre-cell is a ``cylindrical algebraic set,''
i.e.~it is of the form
\begin{equation*}
\Omega= \{(x,r)\colon x\in V_{\Omega},
f_{1,\Omega}(x)<r<f_{2,\Omega}\}
\end{equation*}
for $V_{\Omega}\subset\RR^2$ an open, semi-algebraic set and
$f_{1,\Omega}$, $f_{2,\Omega}$ algebraic functions.

Now, $\bdry(V_\Omega)$ is a 1--dimensional semi-algebraic set, and
thus it can be written uniquely as an almost disjoint finite union of
segments of irreducible algebraic curves such that if any two
segments share a boundary point then their defining polynomials
are distinct (and thus neither defining polynomial divides the
other). We will call the boundaries of these segments the
\emph{vertices} of $V_\Omega$. Now, for each vertex $x_0\in
V_\Omega$, erect the wall
\begin{equation*}
W_{x_0,\Omega}=\{(x,r)\in \Omega\colon x^{(1)}=x_0^{(1)}\}.
\end{equation*}
Finally, if $\gamma$ is a 1--dimensional semi-algebraic set, then
we say that $x_0\in\Gamma$ is a $x^{(1)}$--extremal point if there
exists an open neighborhood $U$ of $x_0$ and an irreducible
algebraic curve $\gamma^\prime$ containing $\gamma\cap U$ such
that $\gamma^\prime\cap U$ is contained in one of the closed half
planes $\{x\colon x^{(1)}\geq x_0^{(1)}\}$ or $\{x\colon
x^{(1)}\leq x_0^{(1)}\}$ (see Figure \ref{extremalPtsFig}).
\begin{figure}[h]\label{extrmalPtsFigure}

\centering
\includegraphics[scale=0.3]{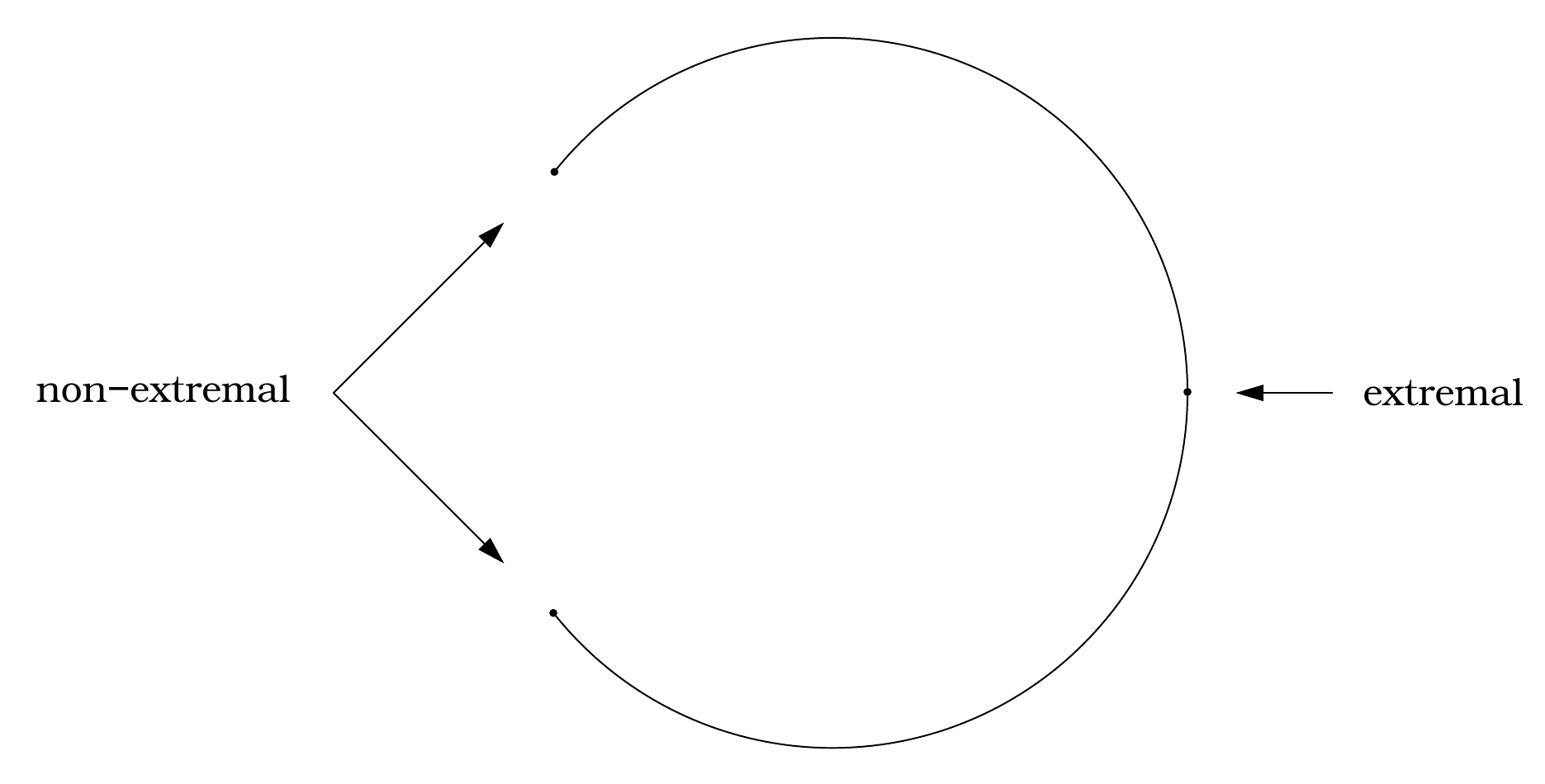}
\caption{Examples of extremal and non-extremal points of a
semi-algebraic curve.}
\label{extremalPtsFig}
\end{figure}
\begin{remark}
This definition of a $x^{(1)}$--extremal point is consistent with
the definition given in Section
\ref{RiemannianMetricIntersectionSection} (Definition
\ref{extremalPtDefinition}) for $\Phi$--conics when $\Phi(x,y)$ is
a smooth algebraic function. The wording of the above definition differs from that of Definition \ref{extremalPtDefinition} since in Definition \ref{extremalPtDefinition} we do not assume that the defining function is algebraic, and thus there is no notion analogous to the Zariski closure of a semi-algebraic set or of an irreducible component of an algebraic set.
\end{remark}

For each extremal point $x_0\in V_\Omega$, erect the vertical wall
$W_{x_0,\Omega}$. Once this has been done, a vertical wall will
have been erected in $\Omega$ above each of the dashed lines in
$V_\Omega$ in Figure \ref{precell_cuttingFig}.
\begin{figure}[h]\label{precellCuttingFigure}
\centering
\includegraphics[scale=0.5]{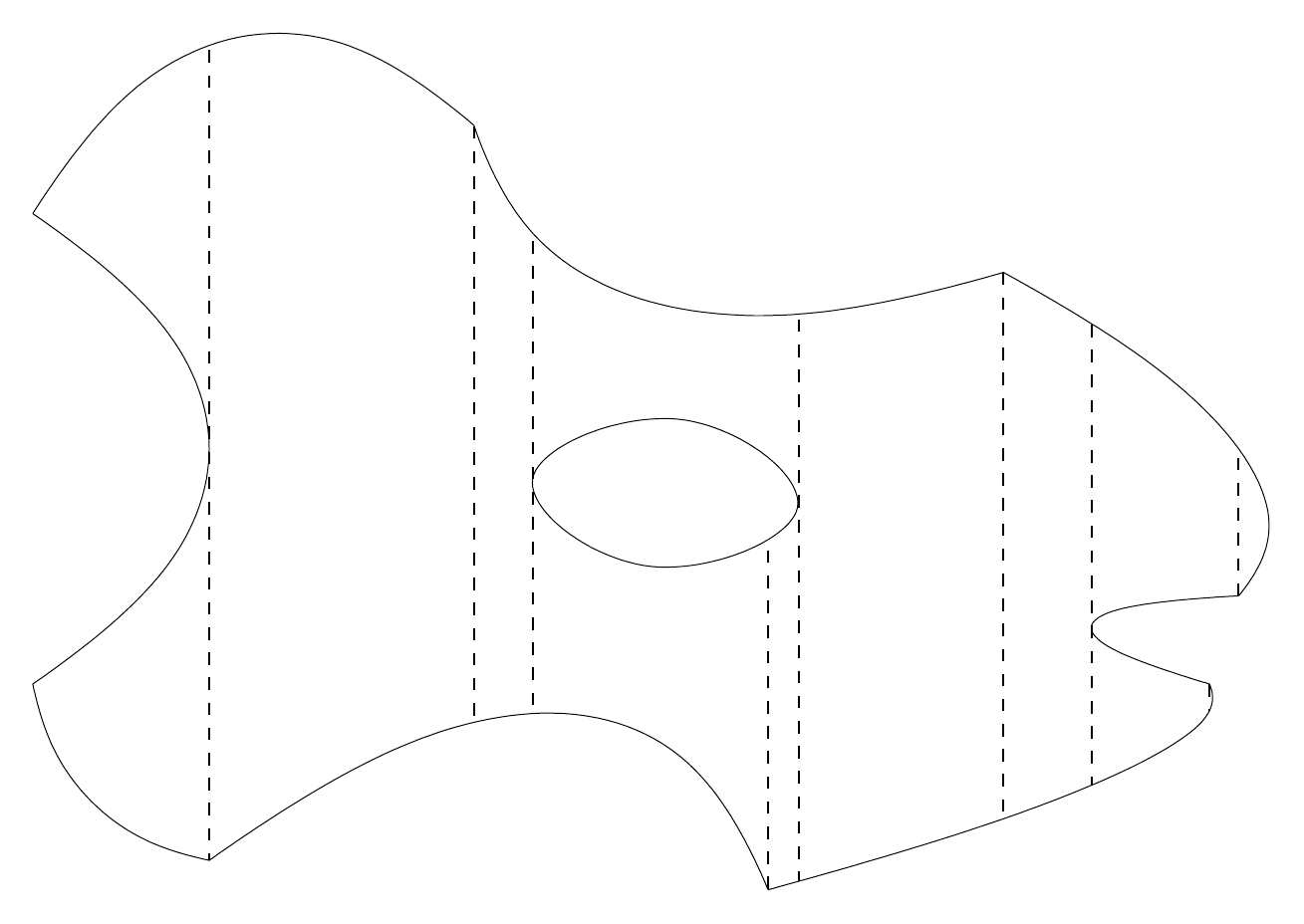}
\caption{A schematic view of $\pi_x(\Omega)$ after vertical walls
have been erected. The dashed lines correspond to vertical walls.}
\label{precell_cuttingFig}
\end{figure}
We also need to add some additional vertical walls
$W_{x_0,\Omega}$ with $x_0$ the endpoint of certain line segments
(since the irreducible algebraic curve that contains a line
segment is of course a line, which (provided it is not parallel to
the $x^{(2)}$--axis) does not have any $x^{(1)}$--extreme points),
but in the interest of brevity we shall gloss over this point (we
can also ensure that line segments never occur by applying a
slight perturbation at an earlier stage of the decomposition).

Once these vertical walls have been erected for each cell
$\Omega$, the resulting arrangement of surfaces partitions $\RR^3$
into topologically trivial open sets (cells). This partition has
the following properties:
\begin{enumerate}[label=(\roman{*}), ref=(\roman{*})]
\item\label{cellDecomItemOne} Each cell is a semi-algebraic set
defined by at most 6 algebraic surfaces.%
\item\label{cellDecomItemTwo} For each cell $\Omega$, there is a
collection of at most 6 surfaces $S_1,\ldots, S_6\in\mathcal S$
such that if the above cell decomposition algorithm were applied
to $\mathcal S^\prime=\{S_1,\ldots,S_6\}$, then $\Omega$ would be
one of the resulting cells in the decomposition.
\item\label{cellDecomItemThree} There are $\lesssim N^3\log N$
cells.
\end{enumerate}
Properties \ref{cellDecomItemOne} and \ref{cellDecomItemTwo} are
immediate from the above cell decomposition algorithm: each cell
$\Omega$ is contained in a unique pre-cell $\Omega^\prime$. The
top and bottom of $\Omega$ are the same algebraic surfaces
$S,\tilde S$ as the top and bottom of $\Omega^\prime$. The
``front'' and ``back'' walls of $\Omega$ (if they exist) are
segments of the vertical wall raised above curves
$\gamma,\tilde\gamma$ which were obtained by intersecting
respectively $S$ and $\tilde S$ with two other surfaces
$S^\prime,\tilde S^\prime\in\mathcal S$, and the ``right'' and
``left'' walls of $\Omega$ (if they exist) are walls of the form
$W_{x_0,\Omega^\prime}$ where $x_0$ is a point of intersection of
$\gamma_1$ and $\gamma_2$, where $\gamma_1$ is a section of $S\cap
S^\prime$ or $\tilde S\cap\tilde S^\prime$, and $\gamma_2$ is a
section of $S\cap S_1$ or $\tilde S\cap \tilde S_1$ for some $S_1$
or $\tilde S_1\in\mathcal S$.

The analysis required to obtain \ref{cellDecomItemThree} is
somewhat lengthy, but the key idea is as follows. The main step in
obtaining Property \ref{cellDecomItemThree} is to bound the number
of vertices in the sets $V_\Omega$, since a bound on the number of
vertices leads to a bound on the number of vertical walls
$W_{x_0,\Omega}$ added to the arrangement (the contribution from
the vertical walls from $x^{(1)}$--extremal points is negligible).
These vertices arise when the algebraic curves defining $\partial
V_\Omega$ intersect. By B\'ezout's theorem, any two algebraic
curves intersect in at most $O(1)$ places (since $\Phi$ is of
bounded degree, all of the algebraic curves appearing in the cell
decomposition are also of bounded degree). This fact allows us to
use the theory of Davenport-Schinzel sequences to control the
\emph{total} number of intersections between the algebraic curve
segments that define the boundaries of the cells (and thus the
total number of vertices occurring in the sets $V_{\Omega}$ as
$\Omega$ ranges over the cells in the decomposition).

Property \ref{cellDecomItemTwo} of the cell decomposition allows
us to use a random sampling argument of the type discussed in
\cite{Clarkson} to obtain Lemma \ref{constDescrComplexCellDecomp}.
We shall give a brief sketch of this lemma here. Let $\mathcal S$
be a collection of 2--dimensional semi-algebraic surfaces with
$\#\mathcal S = n$. Randomly select a subset $\mathcal D\subset S$
with $\#\mathcal D=N < C^{-1}n$ (the requirement $N<C^{-1}n$
allows us to gloss over the distinction between selecting curves
from $\mathcal S$ with and without replacement, since the
probability of the same curve being selected twice is low). Apply
the above cell decomposition algorithm to the collection $\mathcal
D$. For each resulting cell $\Omega$ in the decomposition, let
\begin{equation*}
Z(\Omega)=\#\{S\in\mathcal S\colon S\cap\Omega\neq\emptyset\}.
\end{equation*}

Then,
\begin{equation} \label{probOfMissingAllCircles}
\mathbb{P}\big(Z(\Omega)\geq\lambda\ |\ \Omega\cap S=\emptyset\
\textrm{for all}\ S\in\mathcal D\big)\leq
\Big(1-\frac{\lambda}{n}\Big)^N.
\end{equation}
If we set $\lambda=C\frac{n\log n}{N}$, then the right hand side
of \eqref{probOfMissingAllCircles} is $\lesssim n^{-C}.$ Thus
since our vertical algebraic decomposition gives us an injection
from $\mathcal D^6$ into the collection of all cells arising from
the decomposition of the collection of surfaces $\mathcal D$, and
since each cell in the resulting decomposition does not intersect
any of the surfaces in $\mathcal D$ (since the cells are subsets
of $\RR^3\backslash\bigcup_{S\in \mathcal D} S$), the probability
that even a single cell meets more than $\lambda=C\frac{n\log
n}{N}$ surfaces is at most $C^\prime n^{6-C}$, which we can make
arbitrarily small by choosing $C$ sufficiently large.
\section{Real Algebraic Geometry}\label{realAlgGeoAppendix}%
In this appendix we shall briefly review a few definitions and
theorems from real algebraic geometry. Throughout our discussion,
the base field shall be $\RR$ and all polynomials shall be assumed
to have real coefficients. Unless otherwise noted, all open sets
shall be assumed to be open in the Euclidean topology. Many of the
results discussed below are applicable to any real field but we
shall not pursue this here. Further details on the material
reviewed below can be found in \cite{Bochnak}, \cite{Basu}, and
\cite{Lojasiewicz} (see \cite{Shiota} for an English summary of
the key results we need from \cite{Lojasiewicz}).

\begin{definition}
A set $S\subset\RR^n$ is \emph{semi-algebraic} if
\begin{equation}\label{semiAlgebraicSetDefnEqn}
S=\bigcup_{i=1}^n \{x\colon f_{i,1}(x)=0,\ldots
f_{i,\ell_i}(x)=0,g_{i,1}(x)>0,\ldots,g_{i,m_i}(x)>0\},
\end{equation}
where $\{f_{i,j}\}$ and $\{g_{i,j}\}$ are collections of polynomials.
\end{definition}
\begin{definition}
The \emph{complexity} of a semi-algebraic set is defined as
\begin{equation}
\min \bigg(\sum_{i,j} \deg{f_{i,j}}+\sum_{i,j}
\deg{g_{i,j}}\bigg),
\end{equation}
where the minimum is taken over all representations of $S$ of the
form \eqref{semiAlgebraicSetDefnEqn}.
\end{definition}
\begin{remark}
This definition of complexity is not standard. In the body of the
paper we refer to sets of ``bounded complexity.'' This means that
the complexity of the semi-algebraic set is bounded by a number
that depends only on the defining function $\Phi$ from
\eqref{ML3Bound}.
\end{remark}
\begin{definition}
A function $f\colon \RR^n\to\RR^m$ is semi-algebraic if its graph
is a semi-algebraic set. The complexity of a semi-algebraic
function is the complexity of its graph.
\end{definition}
\begin{theorem}[Tarski-Seidenberg]
Let $S\subset\RR^n$ be semi-algebraic. Then
\begin{equation*}
\pi_{(x_1,\ldots,x_{n-1})}(S)\subset\RR^{n-1}
\end{equation*}
 is
semi-algebraic, and the complexity of
$\pi_{(x_1,\ldots,x_{n-1})}(S)$ is controlled by the complexity of
$S$.
\end{theorem}
\begin{definition}
Let $S\subset\RR^n$ be a semi-algebraic set. We define
\begin{equation}\mathcal I(S) = \{f\in \RR[X_1,\ldots,X_n]\colon
f|_S\equiv 0\}.
\end{equation}
$\mathcal I(S)$ is an ideal in $\RR[X_1,\ldots,X_n]$.
\end{definition}
\begin{definition}
For an ideal $\mathcal I$ in $\RR[X_1,\ldots,X_n]$, we define
\begin{equation} \mathcal
Z(\mathcal I)=\{(x_1,\ldots,x_n)\in\RR^n\colon f(x_1,\ldots,x_n)=0\
\textrm{for all}\ f\in \mathcal I\},
\end{equation}
so in particular, $S\subset\mathcal Z(\mathcal I(S)).$
\end{definition}
\begin{definition}
let $S$ be a semi-algebraic set. We define
\begin{equation*}
\mathcal{P}(S)=\RR [X_1,\ldots,X_n]/\mathcal{I}(S).
\end{equation*}
Then the dimension of $S$ is given by
\begin{equation*}
\dim(S)=\dim(\mathcal{P}(S)),
\end{equation*}
the maximal length of a chain of prime ideals in the ring
$\mathcal{P}(S)$ (see e.g.~\cite{eisenbud}).
\end{definition}
\begin{proposition}
Let $S$ be a semi-algebraic set. Then $S$ has the same dimension
as its closure in the real Zariski topology, i.e.
\begin{equation*}
\dim(S) = \dim(\mathcal Z(\mathcal I(S))),
\end{equation*}
and the latter set is algebraic.
\end{proposition}

\begin{proposition}\label{semiAlgebraicStratification}
Let $f(\underline x, x_{n+1})$ be a polynomial in $n+1$ variables.
Then there exists a partition of $\RR^{n}$ into semi-algebraic
sets $A_1,\ldots,A_m$ and for each $i=1,\ldots,m$ a finite number of
semi-algebraic functions $\xi_{i,1},\ldots,\xi_{i,\ell_i}\colon
A_i\to\RR$ such that
\begin{enumerate}[label=(\roman{*}), ref=(\roman{*})]
\item For each $\underline x\in A_i$ such that $f(\underline
x,\cdot)$ is not identically 0,
\begin{equation}
\{\xi_{i,1}(\underline x),\ldots,\xi_{i,\ell_i}(\underline
x)\}=\{x_{n+1}\colon f(\underline x,x_{n+1})=0\}.
\end{equation}
\item\itemizeEqnVSpacing
\begin{equation}
 \operatorname{graph}(\xi_{i,j})\subset\{f=0\}.
\end{equation}
\end{enumerate}
The complexity of the $A_i$ and $\xi_{i,j}$ depend only on the
complexity of $f$.
\end{proposition}
\begin{corollary}\label{semiAlgebraicDecompCor}
Let $S\subset\RR^{n+1}$ be an algebraic set. Then we can write
\begin{equation}
S=\bigcup_1^n S_i\cup\bigcup_1^mT_i,
\end{equation}
with $S_i=\operatorname{graph}(f_i|_{A_i})$ for $f_i$ a smooth
algebraic function and $A_i\subset\RR^n$ an open semi-algebraic
set, and $\dim\pi_{(x_1,\ldots,x_n)}(T_i)<\dim S$. The complexity
of the $f_i,A_i,$ and $T_i$ depend only on the complexity of $S$.
\end{corollary}
\begin{remark}
In addition to Proposition \ref{semiAlgebraicStratification},
Corollary \ref{semiAlgebraicDecompCor} relies on the the fact that
the set of singular points of a semi-algebraic set is itself a
semi-algebraic set of strictly lower dimension (see Chapter 2 of \cite{Bochnak} for a complete discussion of these ideas).
\end{remark}
\begin{proposition}
Let $S=\bigcup_1^n S_i$ with $S_i$ a semi-algebraic set
homeomorphic to $[0,1]^{d_i}$. Then
$\dim(S)=\max\{d_1,\ldots,d_n\}$.
\end{proposition}
\begin{proposition}\label{smoothManifoldDim}
Let $S$ be a semi-algebraic set that is also a smooth manifold.
Then $\dim(S)$ equals the dimension of $S$ as a smooth manifold.
\end{proposition}

\bibliographystyle{amsplain}

\end{document}